%% file: main.tex
\documentclass[11pt,letterpaper]{article}

\usepackage{amsmath,amssymb,amsthm,comment,bbm}
\usepackage{mathtools}
\usepackage{array}
\usepackage{bbold,bm}
\usepackage{colonequals}
\usepackage{xcolor}
\usepackage[letterpaper,margin=1in]{geometry}
\usepackage{hyperref}
\usepackage[nottoc]{tocbibind}

\hypersetup{
    hidelinks,
    bookmarksopen=true,
    bookmarksnumbered=true,
    pdfpagemode=UseOutlines
}

\definecolor{cblue}{rgb}{0.121569,0.466667,0.705882}
\definecolor{corange}{rgb}{1.000000,0.498039,0.054902}
\definecolor{cgreen}{rgb}{0.172549,0.627451,0.172549}
\hypersetup{
    linkcolor=cblue,
    citecolor=cgreen,
    urlcolor=corange,
    colorlinks=true,
}

\DeclareSymbolFont{sfoperators}{OT1}{cmss}{m}{n}
\DeclareSymbolFontAlphabet{\mathsf}{sfoperators}
\makeatletter
\renewcommand{\operator@font}{\mathgroup\symsfoperators}
\makeatother
\DeclareMathOperator{\Tr}{Tr}
\DeclareMathOperator{\tr}{tr}

\theoremstyle{plain}
\newtheorem{theorem}{Theorem}[section]
\newtheorem{proposition}[theorem]{Proposition}
\newtheorem{corollary}[theorem]{Corollary}
\newtheorem{lemma}[theorem]{Lemma}

\theoremstyle{definition}
\newtheorem{definition}[theorem]{Definition}
\newtheorem{example}[theorem]{Example}

\theoremstyle{remark}
\newtheorem{remark}[theorem]{Remark}

\newcommand{\E}{\mathbb{E}}
\newcommand{\R}{\mathbb{R}}
\newcommand{\C}{\mathbb{C}}
\newcommand{\CC}{\mathbb{C}}
\newcommand{\EE}{\mathbb{E}}
\newcommand{\NN}{\mathbb{N}}
\newcommand{\RR}{\mathbb{R}}
\newcommand{\NNN}{\mathcal{N}}

\newcommand{\sA}{\mathcal{A}}

\newcommand{\sN}{\mathcal{N}}
\newcommand{\sH}{\mathcal{H}}

\newcommand{\sX}{\mathcal{X}}
\newcommand{\sL}{\mathcal{L}}
\newcommand{\sZ}{\mathcal{Z}}

\newcommand{\sa}{\mathsf{sa}}
\newcommand{\free}{\mathsf{free}}
\newcommand{\vect}{\mathsf{vec}}

\newcommand{\Xfree}{\mathcal{X}_{\free}}

\newcommand{\inner}[2]{\langle #1, #2 \rangle} %
\newcommand{\innerL}[2]{\left\langle #1, #2 \right\rangle} %

\newcommand{\sym}{\mathsf{sym}}
\renewcommand{\vec}{\mathsf{vec}}
\newcommand{\Cov}{\mathsf{Cov}}

\newcommand{\brap}[1]{\left( #1 \right)}

\newcommand{\norm}[1]{\left\lVert#1\right\rVert}

\newcommand{\eps}{\varepsilon}
\newcommand\defeq{\vcentcolon=}
\newcommand{\Norm}[1]{\left\lVert#1\right\rVert}

\newtheoremstyle{graytheorem}
  {\topsep}{\topsep}
  {\small\color{gray}\itshape}{}
  {\small\color{gray}\bfseries}{.}{.5em}{}
\theoremstyle{graytheorem}

\theoremstyle{plain}

\definecolor{AfonsoBlue}{RGB}{30,65,123}
\definecolor{TimGreen}{rgb}{0.172549,0.627451,0.172549}
\definecolor{PNNOrange}{RGB}{242,140,40}

\title{Matrix Concentration and Equivalent Operators on Fock Spaces}

\author{%
  \makebox[\textwidth][c]{%
    Afonso S. Bandeira\thanks{Department of Mathematics, ETH Z\"urich. \texttt{bandeira@math.ethz.ch}}
    \qquad
    Dmitriy Kunisky\thanks{Department of Applied Mathematics \& Statistics, Johns Hopkins University. \texttt{kunisky@jhu.edu}}
    \qquad
    Petar Nizi\'{c}-Nikolac\thanks{Department of Mathematics, ETH Z\"urich. \texttt{petar.nizic-nikolac@ifor.math.ethz.ch}}
  }\\[0.8em]
  \makebox[\textwidth][c]{%
    Lucas Pesenti\thanks{Department of Mathematics, ETH Z\"urich. \texttt{lpesenti@ethz.ch}}
    \qquad\qquad
    Robert Wang\thanks{Cheriton School of Computer Science, University of Waterloo. \texttt{robert.wang2@uwaterloo.ca}}
  }
}

\date{October 1, 2026}

\hypersetup{
  pdfauthor={Afonso S. Bandeira, Dmitriy Kunisky, Petar Nizic-Nikolac, Lucas Pesenti, and Robert Wang},
}

\begin{document}
\pagenumbering{roman}

\maketitle
\thispagestyle{empty}

\begin{abstract}
    We develop a method for proving matrix concentration inequalities by identifying random matrices with associated deterministic operators acting on suitable Fock spaces and controlling norms of these operators and their restrictions to low-order subspaces.
    Applying this method, we obtain new strengthenings of the non-commutative Khintchine inequality that improve on the state of the art, in particular sharpening recent inequalities due to Bandeira, Boedihardjo, and van Handel (2023) quantifying intrinsic freeness of random matrices.
    Our proofs of these results are based on relatively simple operator algebra arguments and involve neither Gaussian interpolation nor explicit combinatorics of tracial moments.
    Further, our techniques apply equally well to several models of non-commutative random variables of interest in the literature, such as operator series constructed from $q$-Gaussian and $\Gamma$-independent systems of operators, treating all of these objects with the same method.
    We obtain new norm bounds for such operators both in the style of the non-commutative Khintchine inequality and in the style of Lehner's operator norm formula.
\end{abstract}

\clearpage

\tableofcontents
\thispagestyle{empty}

\clearpage

\pagenumbering{arabic}
\section{Introduction}

\emph{Matrix concentration inequalities} are one of the main tools of modern non-asymptotic random matrix theory, giving general-purpose bounds on norms of random matrices and finding applications in several areas of pure and applied mathematics as well as in physics, computer science, and statistics (see, e.g., the surveys~\cite{T15,VanHandel-2017-StructuredRandomMatrices,bandeira2025randommatrices}).
Our goal in this paper is to develop a new perspective on matrix concentration inequalities that leads to improvements and unifications of previous results.

While our results handle a wide range of models, from random matrices to certain deterministic operators arising in non-commutative probability, we focus for now on the \emph{Gaussian series} random matrix model,
\begin{equation}
\label{eq:gauss-series}
X = \sum_{i = 1}^n g_i A_i\,,
\end{equation}
for $g_i \sim \sN(0, 1)$ i.i.d.\ standard Gaussian random variables and $A_i \in M_d(\CC)_{\sa}$ deterministic self-adjoint $d \times d$ matrix coefficients.
Any self-adjoint random matrix with centered jointly Gaussian entries can be expressed in this way for suitable $A_i$.
We focus on the goal of controlling the expectation of the spectral norm $\EE \|X\|$; tail bounds can be obtained from these with standard tools (see the surveys cited above for details).
Matrices that are not self-adjoint can also be treated by Hermitian dilation.

A classical result in this area is the celebrated Non-Commutative Khintchine (NCK) inequality.

\begin{theorem}[NCK inequality~\cite{LustPiquard1986,NCK}]
    \label{thm:nck}
    In the above setting,\footnote{We write $x\lesssim y$ if $x \leq Cy$ for a universal constant $C$, and $x \asymp y$ if $x \lesssim y$ and $y \lesssim x$.}
    \[ \sigma(X) \lesssim \EE\|X\| \lesssim \sigma(X) (\log d)^{1/2}\,, \]
    where we define the \emph{matrix variance} parameter
    \[
    \sigma(X) \colonequals \left\|\sum_{i = 1}^n A_i^2 \right\|^{1/2} = \| \EE X^2\|^{1/2}\,.
    \]
\end{theorem}
\noindent
Only the logarithmic factor depending on the dimension $d$ in the upper bound precludes the NCK inequality from giving an estimate of $\EE\|X\|$ that is tight up to constants. This factor also corresponds to the dimensional factor appearing in the tail bound of the matrix Bernstein inequality.

A central goal of matrix concentration inequalities is to provide expectation and tail estimates for $\|X\|$, but two standard examples reveal that such estimates cannot be sharp up to constants while depending solely on $\sigma(X)$: for a Wigner matrix, with i.i.d.\ standard Gaussian entries above the diagonal, the logarithmic factor in NCK is known to be spurious (the lower bound of Theorem~\ref{thm:nck} is tight); on the other hand, for a diagonal matrix with i.i.d.\ standard Gaussian entries on the diagonal, the logarithmic factor is required (the upper bound of Theorem~\ref{thm:nck} is tight).

Various refinements of the NCK inequality have been proposed in recent years aiming to clarify this matter, starting with the work of Tropp~\cite{Tropp2018SecondOrder}.
We focus on the following result of \cite{BBvH23}, which makes further progress in this direction using the additional parameter
\begin{equation}
    v(X) \defeq \left\|\sum_{i = 1}^n \vec(A_i)\vec(A_i)^{*}\right\|^{1/2} = \|\Cov(\vec(X))\|^{1/2}.\label{eq:v-covariance}
\end{equation}
The elegant method of that work compares $X$ to a related deterministic operator $\Xfree$, defined as
\begin{equation}
    \Xfree \defeq \sum_{i = 1}^n A_i \otimes s_i\,, \label{eq:Xfree}
\end{equation}
for $s_1, \dots, s_n$ a family of freely independent semicircular operators on a suitable Hilbert space (see Section~\ref{sec:fock-ind}) and satisfying 
\[ \sigma(X) \leq \|\Xfree\| \leq 2\sigma(X)\,. \]
In a sense, $\Xfree$ can be seen as an ``optimistic proxy'' for $X$: if $\EE \|X\| \approx \|\Xfree\|$, then it is the lower bound of the NCK inequality that is tight.
Because $\Xfree$ also corresponds to the summands in the Gaussian series behaving as though they were freely independent, \cite{BBvH23} refer to this situation as \emph{intrinsic freeness} of $X$.
One of the main results of~\cite{BBvH23} gives sufficient conditions for this to be the case.
\begin{theorem}[Intrinsic freeness NCK inequality~\cite{BBvH23}]
    \label{thm:bbvh}
    In the above setting,
    \begin{align}
    \EE\|X\| 
    &\leq \|\Xfree\|\left(1 + C\left(\frac{v(X)}{\sigma(X)}\right)^{1/2}(\log d)^{3/4}\right) \label{eq:bbvh-Xfree} \\
    &\lesssim \sigma(X) + v(X)(\log d)^{3/2}\,. \label{eq:bbvh-nck}
    \end{align}
\end{theorem}  
\noindent
Provided that $v(X) / \sigma(X) \ll (\log d)^{-3/2}$ (say over a sequence of random matrix models with growing dimension $d \to \infty$), we find that $\EE\|X\| \le (1 + o(1))\|\Xfree\|$, a remarkably precise estimate in terms of the deterministic quantity $\|\Xfree\|$.
This norm $\|\Xfree\|$ admits an explicit variational description due to Lehner~\cite{Leh99}, which we discuss in Section~\ref{sec:lehner}; this description may also be formulated as a semidefinite program parameterized by the matrices $A_i$~\cite{Kunisky2026Lehner}.
Further,~\cite{BBvH23} develops a more detailed theory of how, under such assumptions, various other spectral statistics of $X$ resemble those of $\Xfree$.

However, a few aspects of this theory remain lacking.
First, the logarithmic factors above are larger than the $(\log d)^{1/2}$ appearing in the original NCK inequality, and it is natural to expect that such improved NCK inequalities can be refined to have this smaller polylogarithmic dependence.
Second, intuition drawn from the proofs of the NCK inequality and its refinements suggests that the \emph{commutativity} or lack thereof of the $A_i$ should govern whether the polylogarithmic factor should be present or not.
In the theory of \cite{BBvH23}, this is measured by the parameter $v(X)$; however, it is unclear from its definition in what sense this parameter actually measures commutativity.\footnote{In the arguments  in~\cite{BBvH23}, $v(\cdot)$ arises indirectly in bounds on the alignment parameter $w(\cdot)$ of Tropp~\cite{Tropp2018SecondOrder}.}
Finally and relatedly, the parameter $v(X)$ fails a natural invariance: we have
$\|X \otimes I_m\| = \|X\|$ and $\sigma(X \otimes I_m) = \sigma(X)$, while $v$ satisfies
\begin{equation}\label{eq:invariancetensorI}
    v(X \otimes I_m) = v(X) \cdot m^{1/2}\,.
\end{equation}
Thus over a sequence of such models with growing $m$ the estimates of \cite{BBvH23} degrade more and more, and accordingly Theorem~\ref{thm:bbvh} can become, for large enough $m$ in such a construction, inferior even to the original Theorem~\ref{thm:nck}.

Our immediate goal in this paper is to provide a different refinement of the NCK inequality that always improves on both Theorem~\ref{thm:nck} and Theorem~\ref{thm:bbvh}, achieves a tight polylogarithmic factor $(\log d)^{1/2}$, depends on a more natural measurement of commutativity of the $A_i$, and enjoys the above invariance that is failed by $v(X)$ in~\eqref{eq:invariancetensorI}.\footnote{It is worth mentioning that~\cite{BBvH23} shows that for a parameter to satisfy this invariance and sharply capture the spectral norm of Wigner matrices, it must fail to satisfy the triangle inequality, and indeed this is the case for the parameter $\kappa(X)$ we introduce below.}

Our proof technique is quite different from previous work, including \cite{BBvH23}, and is operator-theoretic in nature. 
The general idea is, instead of the random matrix $X$, to consider the \emph{equivalent operator}, denoted $\mathcal{X}$ below, that acts by matrix multiplication on the space of random vectors whose entries are polynomials in the underlying Gaussian variables $g_1, \dots, g_n$.
The equivalent operator, as its name suggests, captures all relevant information about the behavior of $X$; for example, we will show in Section~\ref{sec:equivalentGaussian} that the tracial moments $\EE \tr X^k$ can be computed in terms of quadratic forms with $\mathcal{X}$.
In the basis of Hermite polynomials, the action of $\mathcal{X}$ admits a concise combinatorial description in terms of \emph{creation} and \emph{annihilation operators} on a suitable \emph{Fock space}.
We borrow these useful notions from quantum mechanics and quantum field theory (QFT); see \cite{folland2008quantum,derezinski2013mathematics} for mathematical references from that perspective.
This same structure is used in the standard construction of the operators $s_i$ appearing in the definition of $\Xfree$ and Lehner's calculation of the norm $\|\Xfree\|$, and this connection allows us to relate $\EE\|X\|$ with $\|\Xfree\|$ via $\mathcal{X}$ without appealing to the interpolation argument of \cite{BBvH23}.

The same proof technique applies equally well to other operators constructed from creation and annihilation operators satisfying various \emph{commutation relations}, as studied extensively in the literature on QFT and associated operator algebras.
For instance, since $\Xfree$ can be constructed in this way, our technique gives a simple proof of the upper bound side of Lehner's formula for $\|\Xfree\|$, as well as bounds on the norms of other operator series such as those having \emph{$\Gamma$-independent} operators in place of the $s_i$, as studied in several recent works~\cite{collinsEigenvalueBound,youssefTuran,BKNPW2026-Paper1}.
The most general bounds we prove all have the same variational form of Lehner's formula, and we believe that our approach considerably clarifies this shared structure.

\subsection{Main results}

\paragraph{Matrix concentration}
We first present our main results for random matrices.
Before describing our inequalities, let us introduce our new matrix concentration parameter:
\begin{align}
    \kappa(X)
    &\colonequals \lambda_{\max}\bigl(C(X)\bigr)^{1/2}\,,
    \label{eq:defkappaX} \\
    C(X)
    &\colonequals \sum_{i,j=1}^n (e_i e_j^{*}) \otimes (A_j A_i)
    = \begin{bmatrix}
        A_1^2 & A_2 A_1 & \cdots & A_n A_1 \\
        A_1 A_2 & A_2^2 & \cdots & A_n A_2 \\
        \vdots & \vdots & \ddots & \vdots \\
        A_1 A_n & A_2 A_n & \cdots & A_n^2
    \end{bmatrix}\,.
    \label{eq:defCX}
\end{align}
This may be understood as follows.
Consider the block matrix $R = [A_1 \, A_2 \, \cdots \, A_n] \in M_{d,dn}(\CC)$.
Then, we have $R^{*}R = \sum_{i, j = 1}^n (e_i e_j^{*}) \otimes (A_i A_j)$, and
\[ \lambda_{\max}(R^{*}R) = \|R\|^2 = \lambda_{\max}(RR^{*}) = \sigma(X)^2\,. \]
In contrast, the matrix $C(X)$ whose largest eigenvalue $\kappa(X)^2$ measures is the so-called \emph{partial transpose} of $R^{*}R$, that is obtained by taking the conjugate transpose of every $d \times d$ block, replacing $A_iA_j$ with $A_jA_i$.
The partial transpose operation plays a prominent role in quantum information theory; see, e.g.,~\cite[Section~6.2.3]{Watrous2018} for details.
We emphasize that it can dramatically transform the spectrum of a matrix, for instance mapping a positive semidefinite matrix to an indefinite one, or a low-rank matrix to a full-rank one.
For our purposes, we note that this operation is transparently related to the commutativity of the $A_i$.
At the extreme, if the $A_i$ all commute, then $C(X) = R^{*} R$ and so $\kappa(X) = \sigma(X)$.
Indeed, this is the worst case, and we have:
\begin{proposition}\label{prop:comparisonKappa}
    $\kappa(X) \leq \min\{\sigma(X),v(X)\}$.
\end{proposition}
\begin{proof}
    Write $C = C(X)$.
    Suppose $u \in \CC^{dn}$ has blocks $u_1, \dots, u_n \in \CC^d$.
    Then, we have
    \[ |\langle u, C u\rangle| = \left|\sum_{i, j = 1}^n \langle A_j u_i, A_i u_j \rangle \right| \leq \sum_{i, j = 1}^n \|A_i u_j\|^2 = \sum_{j = 1}^n u_j^{*}\left(\sum_{i = 1}^n A_i^2\right)u_j \leq \sigma(X)^2\|u\|^2\,, \]
    showing $\kappa(X)^2 \leq \|C\| \leq \sigma(X)^2$.
    
    For the other bound, define matrices $B_a \colonequals \sum_{i = 1}^n (u_i)_a A_i \in M_d(\CC)$ for each $a \in [d]$.
    We have $\|B_a\|_F^2 = \|\sum_{i = 1}^n(u_i)_a\vec(A_i)\|^2 \leq v(X)^2 \sum_{i=1}^n |(u_i)_a|^2$.
    Then, we may also bound
    \[ |\langle u, Cu\rangle| = \left|\sum_{a, b, c = 1}^d (B_a)_{bc} \overline{(B_b)_{ac}}\right| \leq \sum_{a, b, c = 1}^d|(B_a)_{bc}|^2 = \sum_{a = 1}^d \|B_a\|_F^2 \leq v(X)^2\|u\|^2\,, \]
    whereby $\kappa(X)^2 \leq \|C\| \leq v(X)^2$.
\end{proof}
\noindent
Thus, the amount by which $\kappa(X)$ is smaller than $\sigma(X)$ is a natural measurement of the non-commutativity of the $A_i$.
Further, $\kappa(X)$ is easily verified to have the invariance suggested earlier:
\begin{proposition}
    $\kappa(X) = \kappa(X \otimes I_m)$.
\end{proposition}
\begin{proof} 
    This is immediate since $C(X \otimes I_m) = C(X) \otimes I_m$.
\end{proof}

Our first main result, a new strengthened version of the NCK inequality, is the following bound, resembling the bound \eqref{eq:bbvh-nck} of \cite{BBvH23} but with $\kappa(X)$ replacing $v(X)$ and with the logarithmic factor sharpened. Its proof is fully contained in Section~\ref{sec:gaussian}.
\begin{theorem}[Strengthened NCK inequality]
\label{thm:strong-NCK}
    In the above setting,
    \[
        \EE\|X\| \leq 2\,\sigma(X) + 6\,\kappa(X)(\log d)^{1/2}.
    \]
\end{theorem}
\noindent

In Section~\ref{sec:cayley}, we give an illustrative example of an application where our bounds give superior estimates to both the NCK inequality (Theorem~\ref{thm:nck}) and the estimates of \cite{BBvH23} (Theorem~\ref{thm:bbvh}).
This is to the \emph{Gaussian Cayley matrices} of \cite{BKMZ24-Cayley}, a group-structured Gaussian series model proposed in that work as a ``challenge case'' for matrix concentration inequalities.
Indeed, we show that the bounds of \cite{BBvH23} behave identically to the NCK inequality in this case, while our improved NCK inequality (Theorem~\ref{thm:strong-NCK}) and the parameter $\kappa(X)$ are sensitive to the particular group structure of such a model and in particular the profile of dimensions of irreducible representations of a group, which may be viewed as a measurement of non-commutativity.

Our tools also provide a bound with the leading term $\|\Xfree\|$, providing an improvement to the bound~\eqref{eq:bbvh-Xfree} of \cite{BBvH23}. 
As in~\cite{BBvH23}, this bound depends on the $v(X)$ parameter.

\begin{theorem}[Strengthened intrinsic freeness NCK inequality]
    \label{thm:strong-intrinsic}
    In the above setting,
\begin{equation}\label{eq:almostprovesBBvHconj}
\EE\|X\| \leq \|\mathcal X_{\mathrm{free}}\|\left(1 + 8\left(\frac{v(X)^2}{\sigma(X)^2}\log d\right)^{1/2} + 8\left(\frac{v(X)^2}{\sigma(X)^2}\log d\right)^{1/3}\right).
\end{equation} 
\end{theorem}
\noindent 
In particular, this implies $\EE\|X\| \le (1 + o(1))\|\mathcal X_\free\|$ whenever $v(X)/\sigma(X) \ll (\log d)^{-1/2}$. 
We also note that, both in~\eqref{eq:bbvh-nck} and~\eqref{eq:almostprovesBBvHconj}, $v(X)$ can be replaced by 
     \(
        v^{\prime}(X) \colonequals \min\{\sigma(X), v(X)\} 
    \)        
by simply using the NCK inequality if $v(X)\geq \sigma(X)$. The proof of Theorem~\ref{thm:strong-intrinsic} can be found in Section~\ref{sec:xfreeBound}.

 The authors of~\cite{BBvH23} posed the question of whether a bound on Gaussian series norms holds of the form $\EE\|X\| \leq \|\mathcal X_{\mathrm{free}}\| + C\,v(X)(\log d)^{1/2}$ for an absolute constant $C \geq 0$, which corresponds to~\eqref{eq:almostprovesBBvHconj} without the last term inside the parentheses.
Theorem~\ref{thm:strong-intrinsic} is a partial answer to this question:~\eqref{eq:almostprovesBBvHconj} both provides the same asymptotic sufficient condition $v(X)/\sigma(X) \ll (\log d)^{-1/2}$ to guarantee $\EE\|X\| \le (1+o(1))\|\Xfree\|$, and gives the same bound (up to constants) when instead we have $v(X)/\sigma(X) \gtrsim (\log d)^{-1/2}$. 
The difference between the claims is only in how quickly the $o(1)$ term in the former case decays.

\paragraph{Lehner-type operator norm bounds}
Lastly, our technique also implies bounds on certain operators of interest in the study of operator algebras beyond their role in random matrix theory.
The bounds we obtain are relatives of Lehner's formula \cite{Leh99} for $\|\mathcal{X}_{\free}\|$ mentioned above.
This formula reads:
\begin{equation}\label{eq:road_lehner}
    \norm{A_0 \otimes \mathbf{1} + \sum_{i=1}^n A_i \otimes s_i}
    =\max_{\eps\in \{\pm 1\}}\inf_{Z\succ0}\lambda_{\max}\brap{\eps A_0+Z^{-1}+\sum_{i=1}^n A_iZA_i}\,,
\end{equation}
where the $s_i$ are freely independent semicircular elements and $\mathbf{1}$ is the identity operator acting on the same Hilbert space as the $s_i$.
We give a new simple proof of the upper bound of this result, which also generalizes to a broad class of $s_i$ as follows.
See Theorem~\ref{thm:variational} for an expanded version.
\begin{theorem}[Generalized Lehner-type norm bound]\label{thm:generalizedLehnerIntro}
    Let $\Gamma \in M_n(\R)_\sa$ with $\Gamma_{ij}\in [-1,1], \Gamma_{i,i} < 1$ for $i,j \in [n]$, $A_0,\ldots, A_n \in M_d(\C)_\sa$, $\ell_1,\ldots, \ell_n$ be the creation operators on the $\Gamma$-Fock space (Definition~\ref{def:gammaFock}), and $\sX = A_0\otimes \mathbf 1 +\sum_{i=1}^nA_i\otimes (\ell_i+\ell_i^*)$. 
    Let $N_{\Gamma}$ be the \emph{number operator} (Definition~\ref{def:Gamma-N}) associated to these operators.
    For each $Z \in M_d(\CC)_{\sa}$ with $Z \succ 0$, define the block matrix
    \[
        C_{\Gamma}(Z) \colonequals [\Gamma_{ij} \cdot A_j Z A_i]_{i, j = 1}^n \in M_{dn}(\C)\,.
    \]
    Then
    \[
        \|\mathcal{X}\| \leq \max_{\eps\in \{\pm 1\}} \inf_{Z\succ 0} \left\{\lambda_{\max} \left(\eps A_0 + \sum_{i=1}^n A_i Z A_i + Z^{-1}\right) + \lambda_{\max}^+(C_\Gamma(Z)) \norm{N_\Gamma}\right\}\,. 
    \]
\end{theorem}
In the case $\Gamma \equiv 0$, the $\ell_i + \ell_i^*$ are freely independent semicircular elements, and so we indeed recover the upper bound of Lehner's formula since $C_{\Gamma}(Z) = 0$ for all $Z$ in this case.
Also, we note that if we have $\Gamma \equiv 1$ then $C_{\Gamma}(I_d)$ is the same as the matrix $C(X)$ for $X = \sum_{i = 1}^n g_i A_i$ featuring in the definition of $\kappa(X)$.
The above does not apply to this case and indeed in this case $N_{\Gamma}$ is unbounded, but the proof technique is entirely analogous to our proof of Theorem~\ref{thm:strong-NCK}.
For other choices of $\Gamma$, we obtain operator bounds relevant to the literature on \emph{$q$-Gaussian systems} (when $\Gamma \equiv q \in [-1, 1]$, as in Example~\ref{ex:qGaussian}) and \emph{$\Gamma$-independent systems} (when $\Gamma$ is the adjacency matrix of a simple graph, i.e.\ has $\Gamma_{i,i} = 0$ and $\Gamma_{i,j} \in \{0, 1\}$, as in Example~\ref{ex:traceMonoid}).

Finally, we note that we do not derive a single master inequality that implies all of our bounds simultaneously, and it is likely that the tools developed here may be useful in combinations beyond those considered in this paper. 
For instance, while improved versions of the matrix Bernstein inequality for sums of independent random matrices (see~\cite[Theorem 1.6.2]{T15}) can be derived immediately by combining our Theorem~\ref{thm:strong-NCK} with the universality principles of~\cite{BvH24,Tropp26}, in a forthcoming work we will establish further sharpened universality principles and Bernstein-type inequalities by operator-theoretic arguments.

\subsection{Proof techniques}

Our proof ideas apply broadly to operator series, or so-called matrix \emph{pencils}, of the form
\begin{equation}
    \label{eq:pencil}
    \sX = \sum_{i = 1}^n A_i \otimes (\ell_i + \ell_i^*)\,,
\end{equation}
for some collection of pairs of adjoint operators $\ell_i, \ell_i^*$ on a Hilbert space $\sH$ and $A_i \in M_d(\CC)_{\sa}$.
As we discuss below in Section~\ref{sec:gaussian}, a random matrix with Gaussian coefficients as above is a special case of this construction.
In this case, the $\ell_i$ are the \emph{creation  operators} and $\ell_i^*$ the associated \emph{annihilation operators} on a \emph{symmetric Fock space}, related through the \emph{canonical commutation relations}.
Alternatively, they can be viewed as operators on polynomials involving multiplications and differentiations associated to the Hermite polynomials. These operators play an important role in QFT, as we will briefly discuss in Section~\ref{sec:gaussian}.

We view such a pencil $\sX$ as an operator \emph{equivalent} to the corresponding Gaussian random matrix $X = \sum_{i = 1}^n g_iA_i$. For example, expected trace
powers of the random matrix can be computed in terms of the operator $\sX$: for all $p \geq 1$,
\[
    \EE \tr X^p = \tr\otimes \tau(\sX^p)\,,
\]
where, for $x$ an operator on $\sH$, $\tau(x) \defeq \langle \Omega,x \Omega\rangle$ for a suitable \emph{vacuum element} $\Omega \in \sH$.
Here $\tr$ is the normalized trace of matrices, so that the tensor product $\tr \otimes \tau$ is a linear functional on $M_d(\CC) \otimes \mathcal{A}$ for suitable algebras $\mathcal{A}$ of operators on $\mathcal{H}$.

Our toolkit centers around bounds in the Loewner ordering on operators of the form \eqref{eq:pencil}.
On the one hand, applying such bounds to the equivalent operators of random matrices readily leads
to bounds on expected norms.
On the other hand, other operators of interest such as $\Xfree$ in the theory of \cite{BBvH23} also have the form \eqref{eq:pencil}, and our techniques apply equally well to those and to several other operators, including the $\Gamma$-independent
and $q$-Gaussian models mentioned above, illuminating the analogies between all these settings.

The basic approach to deducing these bounds on $\sX$ is to partition the pencil into its creation and annihilation  parts:
\begin{equation}\label{eq:decomposition}
    \sX = \sL + \sL^*\,, \qquad \sL = \sum_{i = 1}^n A_i \otimes \ell_i\,.
\end{equation}
To utilize commutation relations between the two parts, we analyze the products $\sL\sL^*$ or $\sL^*\sL$, or variants such as $\sL^* \sZ \sL + \sZ^{-1}$ with an auxiliary weight operator $\sZ \succ 0$.
Expanding such products, in the simplest case $\sL^*\sL$, then leads to products $(A_i \otimes \ell_i)^* (A_j \otimes \ell_j) = (A_i A_j) \otimes (\ell_i^\ast \ell_j)$.
This in turn allows us to use the algebraic structure of the $\ell_i$ to simplify and in particular to take advantage of various notions of non-commutativity.

\subsection{Related work}

\paragraph{Other recent progress on norm bounds} 
The idea of analyzing the norm of an operator by
decomposing it as $\mathcal L+\mathcal L^\ast$ (as in~\eqref{eq:decomposition}) and then exploiting relations
between $\mathcal L\mathcal L^\ast$ and
$\mathcal L^\ast \mathcal L$ has appeared
very recently in two works: in
Caspers and Chen's construction of
strongly convergent matrix models for
$q$-Gaussian systems~\cite[Lemma~3.10]{strongQGaussian}, and in Basu, Kothari
and Midha's bound on the norm of the Sachdev-Ye-Kitaev (SYK) Hamiltonian~\cite[Lemma~4.9]{SYKPravesh}.
Both examples can be viewed as special instantiations of our framework (see Theorem~\ref{thm:q-proj}), corresponding to the particular commutation relations arising in those applications.

Further, we note that the original problem studied in~\cite{SYKPravesh} is itself a random matrix problem, and that their argument similarly proceeds by representing the random matrix by a deterministic operator.
However, the deterministic representation obtained in~\cite{SYKPravesh} is specific to the SYK model, and differs from the natural representation arising through our approach; see Remark~\ref{rem:syk}. The general strategy of representing a random matrix by a deterministic operator and analyzing the latter also appears in Heidary's recent work on the Nelson-Nguyen conjecture~\cite{Heidary-2026-SparseStackOptimalObliviousSubspaceEmbedding}.

\paragraph{$\Gamma$-independent and $q$-Gaussian systems} 
Two important special cases of the general commutation relations considered in this work are $\Gamma$-independent\footnote{Also known as $\varepsilon$-independent or $\Lambda$-free.} semicircular systems~\cite{lambdaFreeness,gammaIndependenceDefinition} and $q$-Gaussian systems~\cite{brownianBozejko,qGaussianSupport,qGaussianReloaded}; see Examples~\ref{ex:qGaussian} and~\ref{ex:traceMonoid}. 
Noncommutative Khintchine inequalities for
$\Gamma$-independent semicircular systems were
established by Collins and Miyagawa~\cite{collinsEigenvalueBound} and
Oliveira Santos, Tripathi, and Youssef~\cite{youssefTuran}.
The argument of~\cite{collinsEigenvalueBound} also uses the decomposition~\eqref{eq:decomposition}.
Building on a bound from~\cite{youssefTuran} (Lemma~\ref{lem:clique-bound}), our Corollary~\ref{cor:Gamma-nck} gives a sharper noncommutative Khintchine inequality in this setting. 
A different application of the same framework yields Corollary~\ref{cor:Gamma-scalar}, which is a key ingredient in the proof of strong convergence of the Kikuchi matrix to a $q$-Gaussian variable in our companion work~\cite{BKNPW2026-Paper1}.

\paragraph{Fock space methods}
The representation of Gaussian random variables in terms of creation and annihilation operators is classical in the literature on Wiener chaos~\cite{Ito-1951-MultipleWienerIntegral} and Fock spaces~\cite{Segal-1956-TensorAlgebrasI}. This is the basis of the approach
that we use in Section~\ref{sec:gaussian} to represent a Gaussian random matrix by an equivalent deterministic operator.
In our setting, these operators appear with matrix rather than scalar coefficients; our main new ingredient is to introduce inequalities in the Loewner order that exploit the noncommutativity of these coefficients (see, e.g., Theorem~\ref{thm:operator-NCK}).

\subsection{Organization of the paper}

The remainder of the paper is organized as follows.
Section~\ref{sec:gaussian} introduces the method of equivalent operators within the more familiar setting of Gaussian random variables and linear operators defined on polynomials.
It proves the strengthened NCK inequality (Theorem~\ref{thm:strong-NCK}) and we hope for it also to serve as an accessible motivation for the role of notions of non-commutative probability in our setting while focusing on classical objects.
Sections~\ref{sec:fock} and~\ref{sec:general_norm_bounds} introduce important notions of non-commutative probability and derive inequalities for several classes of non-commutative variables (realized as operators) using our method.
While these objects are less familiar, we aim for these sections to be accessible to a reader without background in non-commutative probability. 
Section~\ref{sec:appl} develops some applications of these results.

\subsection{Notation}

We use the Kronecker notation $\delta_{ij} = 1$ if $i = j$ and 0 otherwise.
The standard basis of $\RR^d$ or $\CC^d$ is denoted $e_1, \dots, e_d$, where $(e_i)_j = \delta_{ij}$.
We write $\chi = \sum_{i = 1}^d e_i \in \RR^d$ for the all-ones vector; its dimension is clear from context when it appears.
We write $M_d(\mathbb{F})$ for the set of $d \times d$ matrices with entries in a field $\mathbb{F}$, $M_d(\mathbb{F})_{\sa}$ for the set of self-adjoint such matrices, and $M_{d_1, d_2}(\mathbb{F})$ for the set of $d_1 \times d_2$ matrices with entries in $\mathbb{F}$.
For $A \in M_d(\mathbb{C})_{\sa}$, we write $\lambda_{\max}(A)$ for the largest eigenvalue of $A$ and $\lambda_{\max}^+(A) \colonequals \max\{0, \lambda_{\max}(A)\}$.
Further, $\Tr(A) \defeq \sum_{i=1}^d A_{ii}$
and $\tr(A)\defeq \frac 1d \Tr(A)$ will respectively
denote the unnormalized and normalized trace.
We write $I_d \in M_d(\RR)$ for the $d$-dimensional identity matrix, while $\mathbf{1}$ is the identity operator in various Hilbert spaces that are clear from context.
We use calligraphic letters ($\mathcal{X}$, $\mathcal{L}$, ...) for operators on tensor spaces $\CC^d \otimes \mathcal{H}$ for a Hilbert space $\mathcal{H}$, and ordinary letters for other matrices and operators.

\subsection{Acknowledgments and statement on AI}

We thank Joel Tropp for comments on an early draft manuscript including some of the results appearing in this work.

LP acknowledges support from an ETH Z\"urich Postdoctoral Fellowship (Project 26-1 FEL-092).  RW's work on this project was supported by the Queen Elizabeth II Graduate Scholarship in Science and Technology (QEII-GSST). Part of this work was done while RW was an academic guest in the Department of Mathematics at ETH Z\"urich.
PNN and RW thank the Simons Institute for the Theory of Computing and the organizers of the \emph{Spectral Theory Beyond Graphs} program.

AI assistants were used, over the course of several months, in the process of developing the ideas of working with equivalent operators, and later in proofreading the paper. 
Our original motivation came from our efforts to understand Kikuchi matrices through $\Gamma$-independent limits, Feige's hypergraph Moore bound conjecture, and~\cite[Conjecture 9]{bandeira2025randomstrasse101} (see our companion work~\cite{BKNPW2026-Paper1}). 
We remark, however, that several of what we view as the most interesting outcomes of this research effort were obtained only after both Feige's conjecture and~\cite[Conjecture 9]{bandeira2025randomstrasse101} were already solved in~\cite{kothari2025smoothtradeofftensorpca,BKNPW2026-Moore,schmidhuberHastings2026}.

The authors wrote the paper,  verified all of its contents, and all errors are theirs.

\section{Gaussian matrix concentration via Hermite polynomials}
\label{sec:gaussian}

While our technique also provides norm bounds for other operators, including $\Xfree$ and a similar construction from $\Gamma$-independent families, we start by focusing on the Gaussian series model~\eqref{eq:gauss-series} and proving Theorem~\ref{thm:strong-NCK}.

As sketched above, the general premise of our proof technique in this setting is to associate to a \emph{random matrix} a deterministic \emph{equivalent operator}, which in particular captures the combinatorics of expected traces of powers of the matrix.
In general, these operators in this case and the other cases below can be constructed over Fock spaces from creation and annihilation operators, which we introduce in Section~\ref{sec:fock}.
For Gaussian random matrices, however, these operators can be described in a straightforward way as differential operators acting on polynomials, associated to simple identities satisfied by Hermite polynomials and Gaussian integration.
We follow this simplified approach in this section for the sake of exposition.

\subsection{Creation and annihilation operators over Hermite polynomials}\label{sec:prelude}

We first describe the operator-theoretic viewpoint on Hermite polynomials and Gaussian combinatorics in one dimension.
Let $\gamma$ denote the standard Gaussian measure on $\RR$.
We write $\EE$ for expectations with respect to $g \sim \gamma$.
The space of polynomials $\RR[g]$ forms a Hilbert space when equipped with the standard inner product $\inner{f}{h} \colonequals \E[f(g)h(g)]$.
The Hermite polynomials, $\{h_k\}_{k=0}^{\infty}$, form a natural orthonormal basis for $\RR[g]$ under this inner product. 
They can be defined via a three-term recurrence: $h_0 = 1$, and
\begin{equation}\label{eq:3termrecurrenceh_k}
gh_k(g) = \sqrt{k+1}h_{k+1}(g)+\sqrt{k}h_{k-1}(g)\,.
\end{equation}
From these definitions, it follows that expectations can be viewed as inner products with $h_0$, which we will use below:
\[ \EE[f(g)] = \langle f, h_0 \rangle\,. \]

Equation~\eqref{eq:3termrecurrenceh_k} implies that the operation of multiplication by $g$ can be described by the action of two linear operators on $\RR[g]$: one raises the degree of each Hermite polynomial while the other lowers the degree.
We define these operators $\ell,\ell^*:\RR[g]\rightarrow \RR[g]$ by the following action on the Hermite basis: 
\[
    \ell h_k = \sqrt{k+1}h_{k+1}\,,\qquad \ell^\ast  h_k = \sqrt{k}h_{k-1}\,, \qquad \ell^\ast  h_0=0\,.
\]
The operator $\ell$ is called a \emph{creation operator} and $\ell^\ast$ an \emph{annihilation operator}.
The operators $\ell$ and $\ell^\ast$ are indeed adjoint as $\langle \ell h_{k_1}, h_{k_2} \rangle = \langle  h_{k_1}, \ell^\ast h_{k_2} \rangle$ for all $k_1,k_2\geq 0$.

One can also interpret the creation and annihilation operators as differential operators: it follows from standard properties of the Hermite polynomials that $\ell^\ast = \partial_g$ and $\ell = g-\partial_g$, where ``$g$'' denotes the operator of multiplication by $g$ on polynomials.
The fact that they are adjoint corresponds to the Gaussian integration by parts identity:
\[
    \int \partial_g f(g)\, \mathrm d\gamma(g) = \langle \ell^\ast f,h_0 \rangle = \langle f,\ell h_0 \rangle =  \int gf(g) \,\mathrm d\gamma(g)\,.
\]
The three-term recurrence~\eqref{eq:3termrecurrenceh_k} can then be thought of as $g = \ell + \ell^* =  (g - \partial_g) + \partial_g$.

Below we will be interested in expected traces of powers of Gaussian series models.
As a simpler illustration of the same principle, here we consider deriving the moments of a scalar Gaussian random variable.
We do so in terms of a linear functional on operators called the \emph{vacuum state},
\[ \tau(a) \colonequals \inner{h_0}{ah_0}\,. \]
According to the above identification $g = \ell + \ell^*$, this allows us to express the Gaussian moments in terms of the vacuum state and powers of the operator on the right-hand side:
\[
    \E[g^p] = \inner{h_0}{g^p} = \inner{h_0}{(\ell+\ell^*)^ph_0} = \tau((\ell+\ell^*)^p)\,.
\]
For this reason we call $\ell + \ell^*$ an \emph{equivalent operator} to the scalar random variable $g$ (with respect to the state functional $\tau$).

By linearity of the state functional, we may expand the right-hand side in terms of states of products of $\ell$ and $\ell^*$.
The defining feature of $\ell$ and $\ell^*$, also useful in working with such expressions, is their commutation relation.
Observe that their commutator acts by $\ell^\ast \ell h_k - \ell\ell^\ast h_k = (k+1)h_k - k h_k = h_k$. Thus, we have
\[
    \ell^\ast \ell- \ell\ell^\ast = \mathbf{1}\,.
\]
In QFT, this relation is known as the \emph{bosonic commutation relation}~\cite{derezinski2013mathematics}.
As one application, using this relation repeatedly, one may prove the familiar formula for the moments of $g$.\footnote{From the definitions of $\tau, \ell$, and $\ell^*$ we have that, for any operator $a$, $\tau(a\ell^*) = \tau(\ell a) = 0$. The bosonic commutation relation gives $\tau((\ell + \ell^*)^p) = \tau(\ell^\ast (\ell + \ell^*)^{p - 1}) = (p - 1)\tau((\ell + \ell^*)^{p - 2})$ upon commuting $\ell^*$ to the right $p - 1$ times. This gives by induction that $\tau((\ell + \ell^*)^p) = (p - 1)!!$ for $p$ even and zero otherwise.}

\subsection{Multivariate Gaussians and equivalent operators of random matrices}
\label{sec:equivalentGaussian}

The previous construction can easily be extended to a family of independent Gaussian variables. 
Let $\gamma_n$ be the standard Gaussian measure on $\R^n$.
For each $k_1, \dots, k_n$, write $h_{k_1, \dots, k_n}(g_1, \dots, g_n) \colonequals h_{k_1}(g_1) \cdots h_{k_n}(g_n)$.
These are the $n$-variate Hermite polynomials, forming an orthonormal basis of the space of polynomials $\RR[g_1, \dots, g_n]$ under the inner product $\inner{f}{h} = \E[f(g_1,\ldots, g_n)h(g_1,\ldots, g_n)]$.

Under the same convention as before, we associate to every $g_i$ a multiplication operator acting on $\RR[g_1, \dots, g_n]$, and the action of $g_i$ can be decomposed into annihilation and creation parts. Formally, for all $h\in \RR[g_1, \dots, g_n]$, we have
\begin{align*}
    g_ih = (\ell_i+\ell_i^*)h\,,
\end{align*}
where $\ell_i$ and $\ell_i^*$ act on the Hermite basis by
\begin{align*}
    \ell_i h_{k_1,\ldots k_n} &= \sqrt{k_i+1}h_{k_1,\ldots, k_{i}+1,\ldots, k_n}\,,\\
    \ell_i^\ast h_{k_1,\ldots k_n} &= \begin{cases}
        \sqrt{k_i}h_{k_1,\ldots, k_{i}-1,\ldots, k_n} & \text{ if } k_i \geq 1, \\
        0& \text{ if } k_i=0.
    \end{cases}
\end{align*}
These operators satisfy the \emph{multivariate bosonic commutator relations}, given by
\begin{equation}\label{eq:canonical_cr_new}
    \ell_i^* \ell_j- \ell_j\ell_i^* = \delta_{i,j} \mathbf{1}\,.
\end{equation}
These relations, together with the vacuum state, recover joint moments of the Gaussian variables:
\begin{align}
    \E[g_{i_1}g_{i_2}\ldots g_{i_p}] = \tau\left(\prod_{j=1}^p(\ell_{i_j}+\ell_{i_j}^*)\right) = \innerL{h_{0}}{\prod_{j=1}^p(\ell_{i_j}+\ell_{i_j}^*)h_0}\label{eq: moments-multi}
\end{align}
where $h_0 = 1$ is the constant function. 
As before, this can be verified to recover Wick's formula for the left-hand side.
An equivalent operator of any Gaussian linear combination, $\sum_i a_ig_i$ for $a_i \in \R$, is therefore given by the operator $\sum_i a_i (\ell_i+\ell_i^*)$. 
In QFT, this is known as a \emph{Segal field operator}.\footnote{We omit the conventional factor of $1/\sqrt{2}$.}

To study a self-adjoint Gaussian random matrix, $X=\sum_{i=1}^ng_iA_i$, we define an equivalent operator $\sX$ in the space $M_d(\C)$ tensored with the space of linear operators over $\RR[g_1, \dots, g_n]$:
\[
    \sX \colonequals \sum_{i=1}^nA_i\otimes (\ell_i+\ell_i^*)
    = \sL + \sL^\ast,
\]
where $\sL = \sum_i A_i\otimes \ell_i$. As we will see below, the seemingly innocuous decomposition $\sX = \sL + \sL^\ast$ plays a central role in our arguments.

We note that these operators act on the space $\CC^d \otimes \RR[g_1, \dots, g_n]$, which may be viewed as $d$-dimensional vectors, each of whose entries is a polynomial in $g_1, \dots, g_n$.
By \eqref{eq: moments-multi}, $\sX$ then captures the trace moments of $X$ via
\begin{align}
    \E \tr X^p = \tr\otimes \tau(\sX^{p}) = \tr\otimes\tau((\sL + \sL^{\ast})^p)
    \label{eq:trace-moments}
\end{align}
where $\tau(\mathcal{A}) = \inner{h_0}{\mathcal{A} h_0}$ is the vacuum state and $\tr$ is the normalized trace, so that $\tr \otimes \tau$ is a state on $M_d(\C)$ tensored with the space of linear maps on $\RR[g_1, \dots, g_n]$.

\subsection{Proofs of NCK-type inequalities}

Our goal is to bound $\EE\|X\|$ by bounding the traces on the left-hand side.
One may show that, for a bounded operator $\mathcal{T}$, we have $\left|\tr \otimes \tau(\mathcal{T})\right| \leq \|\mathcal{T}\|$, so one might hope to apply this bound on the right-hand side.
While versions of such a strategy are sensible in some applications we discuss later, in the Gaussian setting operators like $\sX$ and $\sL$ are unbounded. 
Instead, we bound these and related operators in the \emph{Loewner order} by simpler operators, and then consider the restrictions of those operators to finite-dimensional subspaces of $\CC^d \otimes L^2(\gamma_n)$ spanned by low-degree polynomials, on which they are bounded.

A crucial role in the bounds in Loewner order will be played by what is referred to in QFT as the \emph{number operator},
\begin{equation}\label{eq:numberoperator}
    N = \sum_{i=1}^n \ell_i\ell_i^*.
\end{equation}
The number operator acts on Hermite polynomials by $Nh_{k_1,\ldots k_n} = (k_1+\ldots +k_n)h_{k_1,\ldots, k_n}$. 
In other words, it  scales each Hermite polynomial by its total degree.\footnote{In the QFT formulation over Fock spaces, the degree of a Hermite polynomial corresponds to a number of ``particles'', in which case the number operator counts the number of particles, justifying its name.}

We note that, since Hermite polynomials can have arbitrary degree, the number operator is unbounded, like $\mathcal{X}$ and $\mathcal{L}$.
But, the number operator is bounded on the space of polynomials of degree bounded by some $p$; indeed, it is clear from the above description that this restriction has norm at most $p$.
Thus the number operator is easy to use in the above restriction strategy.

Let us start by deriving bounds on some related operators depending on the number operator~\eqref{eq:numberoperator}.
The parameters of a Gaussian series model $X = \sum_{i = 1}^n g_i A_i$ given earlier in Theorem~\ref{thm:nck} and~\eqref{eq:defkappaX} will appear in our statement, which we recall are:
\begin{align*}
    \sigma(X)^2 = \left\|\sum_{i = 1}^n A_i^2\right\|\,, \qquad
    \kappa(X)^2 = \lambda_{\max}(C(X))\,, \qquad
    C(X) = \sum_{i, j = 1}^n (e_i e_j^{*}) \otimes (A_j A_i)\,.
\end{align*}
As we will see momentarily, the following are direct analogs of the NCK inequality and our improvement in Theorem~\ref{thm:strong-NCK} for operators in the Loewner ordering, and readily give these NCK-type inequalities for $\EE\|X\|$ after a short additional argument truncating these operators to spaces of polynomials of bounded degree, on which they are bounded operators.
The Loewner ordering is convenient here because it allows us to state related inequalities without explicit truncation.
We note also that these bounds may be viewed as variants of ones for scalar linear combinations from the QFT literature; see, e.g., \cite{Cook1953} and \cite[Lemma~2.1]{RodnianskiSchlein2009}.

\begin{theorem}[Loewner NCK-type inequalities]\label{thm:operator-NCK}
    Let $X=\sum_{i=1}^n g_iA_i$ be a self-adjoint Gaussian random matrix. For 
    $\sL = \sum_{i = 1}^n A_i\otimes \ell_i$ we have
    \begin{align*}
        \sL\sL^*&\preceq \sigma(X)^2\,I_d \otimes N\,,\\
        \sL^*\sL&\preceq \left(\sum_{i = 1}^n A_i^2\right)\otimes \mathbf{1}+\kappa(X)^2\,I_d\otimes N\,.
    \end{align*}
\end{theorem}

\begin{proof}
Define the row flattening of the matrix coefficients $R = [A_1 \, \cdots \, A_n]$. Note that $\|R^\ast R\|=\|RR^\ast\|=\sigma(X)^2$ and that we can write
\begin{equation*}
 \sL \sL^* 
  =  \begin{bmatrix}I_d \otimes \ell_1 & \cdots & I_d \otimes \ell_n\end{bmatrix} \left( R^\ast R \otimes \mathbf{1}\right) \begin{bmatrix} I_d \otimes \ell_1^* \\ \vdots \\ I_d \otimes \ell_n^*\end{bmatrix}  
\preceq \sigma(X)^2 \, \sum_{i = 1}^n I_d\otimes \ell_i \ell_i^*  = \sigma(X)^2 I_d \otimes N,
\end{equation*}
giving the first inequality. For the second inequality, we use the commutation relations~\eqref{eq:canonical_cr_new}:
\begin{align*}
\sL^*\sL &= \sum_{i,j = 1}^n A_iA_j\otimes \ell_i^*\ell_j \\
&= \sum_{i,j = 1}^n A_iA_j\otimes (\delta_{ij}\mathbf{1}+\ell_j\ell_i^*)\\
&= \sum_{i = 1}^n A_i^2 \otimes \mathbf{1} + \begin{bmatrix}I_d\otimes \ell_1 &  \cdots  & I_d \otimes \ell_n
\end{bmatrix} (C(X)\otimes \mathbf{1})\begin{bmatrix} I_d \otimes \ell_1^* \\ \vdots \\ I_d \otimes \ell_n^* \end{bmatrix}\,.
\end{align*}
Since the diagonal blocks of $C(X)$ are $A_i^2$, we have $\lambda_{\max}(C(X))\geq 0$ and so the second summand can be upper bounded in Loewner order by $\kappa(X)^2 I \otimes N$, completing the proof.
\end{proof}

\begin{remark}
    We emphasize the key point above that the two orders in which $\mathcal{L}$ and $\mathcal{L}^*$ can be multiplied admit very different bounds in Loewner order.
    It is precisely from working with the second choice $\mathcal{L}^*\mathcal{L}$ that we obtain a superior inequality, and it is precisely the use of the commutation relations in working with this product that we arrive at the partially-transposed matrix $C(X)$.
\end{remark}

We now proceed to bound the norm of the number operator when restricted to low degree polynomials.
Below we write $\Pi_p$ for the orthogonal projection in $\RR[g_1, \dots, g_n]$ to polynomials of degree at most $p$.

\begin{proposition}\label{prop:bosonicNbound}
    Let $N = \sum_i \ell_i\ell_i^\ast$ denote the bosonic number operator as defined in~\eqref{eq:numberoperator}.
    Then,
    \[ \Pi_p N \Pi_p \preceq p\Pi_p\,. \]
\end{proposition}
\begin{proof}
    Both $N$ and $\Pi_p$ are diagonalized by the orthonormal Hermite basis, and further $\Pi_pN\Pi_p$ and $\Pi_p$ have finite rank since the space of polynomials of degree at most $p$ has finite dimension. 
    All Hermite polynomials whose degree is larger than $p$ are in the kernel of both $\Pi_pN\Pi_p$ and $p\Pi_p$.
    The eigenvalues associated to Hermite polynomials whose degree is some $\Delta\leq p$ are $p$ for $p\Pi_p$ and $\Delta$ for $\Pi_pN\Pi_p$, and the result follows.
\end{proof}

We are now ready to prove the main inequalities of this section, the upper bounds of the original NCK inequality (Theorem~\ref{thm:nck}) and of our improved NCK inequality (Theorem~\ref{thm:strong-NCK}).
Both come from the following bounds on tracial moments.
\begin{corollary}\label{cor:nck-projection}
In the setting of Theorem~\ref{thm:operator-NCK} we have the following trace moment bounds for $p\geq 1$:
    \begin{align*}
        \left(\E\tr(X^{2p})\right)^{\frac1{2p}} &\leq 2\sqrt{p}\,\sigma(X)\,, \\
        \left(\E\tr(X^{2p})\right)^{\frac1{2p}}&\leq 2\sqrt{\sigma(X)^2 + (p-1)\kappa(X)^2}\,.
    \end{align*} 
\end{corollary}
\begin{proof}
We define the truncated operators
\begin{align*}
    \sL_p &\colonequals (I_d \otimes \Pi_p)\sL(I_d \otimes \Pi_p)\,, \\
    \sX_p &\colonequals (I_d\otimes \Pi_p)\sX(I_d\otimes \Pi_p) = \sL_p + \sL_p^\ast\,.
\end{align*}
Note that since $\sL^*$ lowers the degree we have $\sL^*(I_d\otimes \Pi_p)=(I_d\otimes \Pi_p)\sL^*(I_d\otimes \Pi_p)=\sL_p^*$. Similarly, since $\sL$ raises the degree we have $\sL_p = (I_d\otimes \Pi_p)\sL(I_d \otimes \Pi_p) = \sL(I_d \otimes \Pi_{p-1})$.

Using Theorem~\ref{thm:operator-NCK} and Proposition~\ref{prop:bosonicNbound}, we obtain
\begin{align*}
     \sL_p\sL_p^*
     &= (I\otimes \Pi_p) \sL \sL^* (I\otimes \Pi_p) \\
     &\preceq \sigma(X)^2 I\otimes \Pi_pN\Pi_p \preceq p\,\sigma(X)^2 I\otimes \Pi_p\,, \\
     \sL_p^* \sL_p 
     &= (I\otimes \Pi_{p-1}) \sL^* \sL (I\otimes \Pi_{p-1}) \\ &\preceq \sum_i A_i^2 \otimes \Pi_{p-1} + \kappa(X)^2 I\otimes \Pi_{p-1}N\Pi_{p-1}\\ &\preceq \left(\sigma(X)^2 +(p-1)\kappa(X)^2\right)(I\otimes \Pi_{p-1})\,.
\end{align*}
Since projections have norm at most 1, we have, by triangle inequality,
\begin{equation}\label{eq:boundonXpBoson}
    \norm{\sX_p} \leq 2\norm{\sL_p} \leq 2\min \left\{ \sqrt{p}\,\sigma(X), \sqrt{\sigma(X)^2 +(p-1)\kappa(X)^2}\right\}\,.
\end{equation}

Lastly, we relate $\EE \tr X^{2p}$ with $\|\sX_p\|$.     Since each multiplication by $\sX$ raises the degree of a polynomial by at most $1$ (and $h_0$ has degree $0$), we have that
    $\sX^p(e_j\otimes h_0) = \sX_p^{p}(e_j\otimes h_0)$ for all $j\in [d]$. 
    Thus, by \eqref{eq:trace-moments}, we have
\begin{align*}
\EE\tr X^{2p}=\tr\otimes \tau(\sX^{2p}) = \frac1d\sum_{j=1}^d \|\sX_p^p(e_j\otimes h_0)\|^2 \leq \norm{\sX_p}^{2p}\,,
\end{align*}
which together with~\eqref{eq:boundonXpBoson} establishes the trace inequalities. 
\end{proof}

Using this, the proofs of the two NCK-type inequalities are immediate.
\begin{corollary}[Concrete upper bound of Theorem~\ref{thm:nck}]
    For $X$ as in \eqref{eq:gauss-series}, we have
    \[ \E\norm{X} \leq 4\sqrt{1+\log d}\,\sigma(X)\,. \]
\end{corollary}
\begin{proof}
    By Jensen's inequality,
\begin{equation}
\EE\|X\| \leq \left(\EE\|X\|^{2p}\right)^{\frac1{2p}} \leq \left(d\, \EE\tr X^{2p}\right)^{\frac1{2p}} = d^{\frac1{2p}} \left(\EE\tr X^{2p}\right)^{\frac1{2p}}\,. \label{eq:norm-trace}
\end{equation}
    The result then follows by choosing $p=\lceil{\log d\rceil}$, so that $d^{1 / 2p} \leq \sqrt{e} \leq 2$. 
\end{proof}
\noindent
Our improvement follows by a slight variation of the same strategy.
\begin{proof}[Proof of Theorem~\ref{thm:strong-NCK}]
    We again start with the bound of \eqref{eq:norm-trace}, which may be written as
    \[ \EE\|X\| \leq \inf_{p \in \NN} 2d^{1/2p}\sqrt{\sigma(X)^2 + p \kappa(X)^2}\,. \]
    The result then follows by the result of Lemma~\ref{lem:optimization-scalar} on this optimization and simplifying using that $\log(2d) \leq 2\log d$.
\end{proof}

In Section~\ref{sec:xfreeBound}, we will prove a strengthening of the estimate of Theorem~\ref{thm:strong-NCK} replacing the leading term $2\sigma(X)$ by $\|\Xfree\|$, at the cost of a different error term depending on $v(X)$ instead of $\kappa(X)$.
We will do this by introducing a variational component to the argument, akin to the proof of Lehner's formula~\cite{Leh99}.

\begin{remark}[Scalar Khintchine inequalities]
    We note that essentially the same argument, taking $d = 1$ or equivalently considering the Segal field operator $\sum_{i = 1}^n a_i(\ell_i + \ell_i^*)$ for scalar coefficients $a_i \in \RR$, gives the classical Khintchine inequalities for linear combinations of i.i.d.\ random variables, such as Rademacher variables~\cite{Khintchine1923}.
    The only difference is that, in this case, one must work with a number operator associated with the orthogonal polynomials of this distribution rather than the Hermite polynomials.
    Thus the number operator captures all of the salient properties of the law of the creation and annihilation operators for such bounds, just as in our argument above.
\end{remark}

\section{Preliminaries on Fock spaces and commutation relations}
\label{sec:fock}

In this section, we introduce a more general and abstract framework for similar calculations to the ones in the previous section that will generalize the algebraic properties of orthogonal polynomials (including Hermite polynomials).
We use the framework of non-commutative probability; many of the relevant definitions also appear prominently in applications to QFT.

A \emph{non-commutative probability space} is a pair $(\sA,\tau)$, where $\sA$ is a unital $\ast$-algebra (interpreted as an algebra of random variables but often constructed as an operator algebra), and $\tau: \sA \rightarrow \CC$ is a linear functional called a \emph{state} (interpreted as a trace or expectation) satisfying $\tau(\mathbf 1) = 1$ and $\tau(a^\ast a)\ge 0$.
The \emph{law} of a self-adjoint element $x \in \sA$ refers to the collection of \emph{moments} $\tau(x^k)$.
We say that the law is a certain real-valued probability measure $\mu$ if these moments are the moments of $\mu$:
\[ \tau(x^k) = \int \lambda^k\,\mathrm d\mu(\lambda) \text{ for all } k \geq 1. \]

The creation and annihilation operators $\ell_i$ and $\ell_i^*$ discussed above can be viewed, in addition to their concrete realization over Hermite polynomials, as elements of a suitable non-commutative probability space.
The bosonic commutation relations that these satisfy correspond to classical independence of scalar Gaussian random variables.
Other forms of non-commutative independence, like free independence, can be realized by creation and annihilation operators satisfying different commutation relations (indeed, this is how one defines the operator $\Xfree$ mentioned above).
In this section, we give constructions of these operators on Hilbert spaces, which provide another way to interpret our previous use of $\RR[g_1, \dots, g_n]$ using the Hermite polynomial basis.

\subsection{The full Fock space}
\label{sec:fock-ind}

We begin by reviewing the construction of freely independent semicircular random variables that appeared for example as the $s_i$ in $\Xfree$ in~\eqref{eq:Xfree}, following~\cite[Lecture~7]{NS06}. 
The \emph{full Fock space} over $\CC^n$ is defined as the space $\mathcal F(\CC^n) \defeq \bigoplus_{k\geq0}(\CC^n)^{\otimes k}$, equipped with the standard Hilbert space inner product, where we take by convention $(\CC^n)^{\otimes 0} = \CC \Omega$ for $\Omega$ a distinguished unit vector called the \emph{vacuum} vector.
Algebras of operators on $\mathcal{F}(\CC^n)$ are given the structure of a non-commutative probability space by endowing them with the \emph{vacuum state}, the functional
\begin{equation}
\label{eq:law}
\tau(x) \colonequals \langle\Omega,x\Omega\rangle.
\end{equation}

We call each $(\CC^n)^{\otimes k}$ the \emph{$k$-particle space} and the span of $\Omega$ the \emph{0-particle space}, this terminology coming from the following interpretation.
Let $e_1,\ldots,e_n$ be an orthonormal basis of $\CC^n$. 
The Fock space has a natural orthonormal basis $\{\Omega\} \cup \{e_{i_1}\otimes \cdots \otimes e_{i_k}:i_1,\ldots, i_k\in [n], k\in \NN\}$.
We can view this basis as the set of all words over the alphabet $\{e_1,\ldots, e_n\}$, where $\Omega$ represents the empty word.
The creation operator $\ell_i$ adds a letter to the left of a word:
\begin{equation}
    \ell_i\Omega = e_i,
    \qquad
    \ell_i(e_{i_1}\otimes\cdots\otimes e_{i_k})
    =e_i\otimes e_{i_1}\otimes\cdots\otimes e_{i_k}.\label{eq:creation}
\end{equation}
Its adjoint $\ell_i^*$ removes the first letter if it is $e_i$ and otherwise annihilates the word:
\begin{equation}
    \ell_i^*\Omega=0,
    \qquad
    \ell_i^*(e_{i_1}\otimes\cdots\otimes e_{i_k})
    =\langle e_i,e_{i_1}\rangle\, e_{i_2}\otimes\cdots\otimes e_{i_k}.\label{eq:anni}
\end{equation}
Then the operators $\ell_i$ and $\ell_j^*$ satisfy the \emph{free commutation relations}:\footnote{Also known as the \emph{Cuntz-Toeplitz relations}, defined to study $C^*$-algebras generated by isometries~\cite{Cuntz1977}.}
\begin{equation}\label{eq:free_cr_new}
    \ell_i^*\ell_j= \delta_{ij}\mathbf{1}.
\end{equation}
With respect to the vacuum state, the operators
\begin{equation*}
    s_i=\ell_i+\ell_i^*,\qquad i=1,\ldots,n,
\end{equation*}
form a family of standard freely independent semicircular random variables~\cite[Corollary~7.17]{NS06}.
Each $s_i$ has the semicircle law normalized to have variance 1 (in the sense of \eqref{eq:law} above).
Freeness means that a product of centered polynomials has state zero whenever adjacent factors involve different variables.
  Precisely, for every $m\geq1$, $\tau\bigl(p_1(s_{i_1})\cdots p_m(s_{i_m})\bigr)=0$
  whenever $p_1,\ldots,p_m\in\CC[x]$ and $i_1,\ldots,i_m\in[n]$ satisfy
  $\tau(p_j(s_{i_j}))=0$ for all $j$ and $i_j\neq i_{j+1}$ for $1\leq j<m$.\footnote{Voiculescu made this definition to study free products of operator algebras~\cite{voiculescu1985symmetries}, using the full Fock space construction. The connection with random matrices came later via the related notion of asymptotic freeness~\cite{Voi91}.}
The vacuum state satisfies the following properties for all $a, b$ in the $C^*$-algebra generated by the $s_i$~\cite[Proposition~7.18]{NS06}:
\begin{enumerate}
    \item $\tau(ab) = \tau(ba)$.
    \item $\tau(a^*a)\geq 0$, with $\tau(a^*a) = 0$ if and only if $a=0$.
\end{enumerate}

\subsection{General commutation relations}

So far, we have seen two examples of
commutation relations that operators
$\ell_1, \ldots, \ell_n$ may satisfy: 
the bosonic commutation relations~\eqref{eq:canonical_cr_new} and
the free commutation relations~\eqref{eq:free_cr_new}.
The following definition generalizes both of these.

\begin{definition}
    \label{def:gamma-CCR}
    Given any undirected weighted graph $\Gamma$ on vertex set $[n]$ with edge weights $\Gamma_{i,j}\in [-1,1]$,
    the $\Gamma$-\emph{canonical commutation relations} $(\Gamma$-CCR)
    are
    \begin{align}
        \ell_i^* \ell_j-\Gamma_{i,j}\ell_j\ell_i^* = \delta_{i,j}\mathbf{1}.\label{eq:ccr}
    \end{align}
    We allow this graph to have loops, so the diagonal weights $\Gamma_{i,i}$ are also defined.
\end{definition}
\noindent The bosonic
and free relations correspond to the
cases $\Gamma\equiv 1$ and $\Gamma\equiv 0$, respectively.
Bożejko and Speicher~\cite{proofGammaRepresentation}
show how to construct operators satisfying the
$\Gamma$-CCR using a different inner
product on the full Fock space $\mathcal F(\CC^n)$. 

\begin{definition}[$\Gamma$-Fock space~\cite{proofGammaRepresentation}]\label{def:gammaFock}
    Define $\langle \cdot,\cdot\rangle_{\Gamma}$ on $\mathcal F(\C^n)$ as follows:
    \begin{enumerate}
        \item Set $\langle \Omega,\Omega\rangle_\Gamma=1$.
        \item For any $\xi\in (\CC^n)^{\otimes k},\xi'\in (\CC^n)^{\otimes k'}$ for $k\neq k'$, set $\langle \xi, \xi'\rangle_\Gamma=0$.
        \item Otherwise,
    for every $\xi \in (\CC^n)^{\otimes {k-1}}$ and $i,j_1, \ldots, j_k\in [n]$, define recursively
    \[
        \langle e_i \otimes \xi ,e_{j_1}\otimes \cdots \otimes e_{j_k} \rangle_\Gamma = \sum_{\substack{r \in [k] \\ j_r=i}}  \langle \xi, e_{j_1}\otimes \cdots \otimes e_{j_{r-1}} \otimes e_{j_{r+1}} \otimes \cdots \otimes e_{j_k}\rangle_{\Gamma} \prod_{s=1}^{r-1} \Gamma_{i,j_s}\,.
    \]
    \end{enumerate}
    After quotienting $\mathcal F(\C^n)$ by the null space of $\langle\cdot,\cdot\rangle_\Gamma$ and taking the completion, we obtain a Hilbert space
    $\mathcal F_\Gamma(\CC^n)$ known as the {\em $\Gamma$-Fock space}.
    On the $\Gamma$-Fock space, the 
    creation operator $\ell_i$ from~\eqref{eq:creation}
    and its corresponding adjoint operator $\ell_i^*$ for $\langle \cdot,\cdot\rangle_{\Gamma}$ 
    satisfy the
    $\Gamma$-CCR (note that $\ell_i^\ast$ is different from the operator in \eqref{eq:anni}). We endow them
    with the vacuum state $\tau(x) = \langle \Omega, x\Omega\rangle$.
\end{definition}

\noindent To interpret this definition, we specialize it to several important special cases.

\begin{example}[Freely independent semicirculars]
    When $\Gamma\equiv 0$, the inner
    product $\langle \cdot, \cdot \rangle_\Gamma$ coincides with the standard inner product on the full
    Fock space,
    so we recover the construction of
    freely independent semicircular
    random variables from Section~\ref{sec:fock-ind}.
\end{example}

\begin{example}[Classically independent Gaussians]
    When $\Gamma\equiv 1$, the $k$-particle space is obtained from $(\CC^n)^{\otimes k}$ by identifying tensors that differ only by a permutation of the indices $\{1,\ldots,k\}$. This in turn
    is equivalent to the construction in
    Section~\ref{sec:equivalentGaussian}
    by associating (up to a normalization) the Hermite polynomial
    $h_{k_1, \ldots, k_n}$ with the 
    symmetric Fock space basis
    vector that contains $k_i$ copies of $e_i$. The operators
    $(\ell_1+\ell_1^\ast,\ldots,\ell_n+\ell_n^\ast)$ therefore have the joint law with respect to the vacuum state of classically independent standard Gaussian random variables.
\end{example}

\begin{example}[$q$-Gaussian system]\label{ex:qGaussian}
    When $\Gamma\equiv q$ for some
    $q\in [-1,1]$, the operators
    $(\ell_1+\ell_1^\ast,\ldots,\ell_n+\ell_n^\ast)$ form a
    {\em $q$-Gaussian system}, an interpolation between the above examples studied, for instance, by the works~\cite{frischBourret,brownianBozejko,qGaussianSupport,speicherCLT,qGaussianReloaded,BKNPW2026-Paper1}.
    Marginally, each $\ell_i+\ell_i^\ast$ has the $q$-Gaussian distribution, which we describe in Section~\ref{sec:deformedNCK}.
\end{example}

\begin{example}[$\Gamma$-independent semicircular family]\label{ex:traceMonoid}
    Suppose $\Gamma$ is a simple unweighted
    graph, i.e., $\Gamma_{i,j}\in \{0,1\}$ for all $i\neq j$, and $\Gamma_{i,i}=0$. We can interpret each original Fock basis vector as a word over the characters $[n]$. Quotienting the null space of $\inner{\cdot}{\cdot}_\Gamma$ induces an equivalence relation between basis vectors, where two words are identified whenever one
    can be obtained from the other by
    repeatedly swapping adjacent letters
    that are adjacent in $\Gamma$. Indeed, pairs of words in the same equivalence class have unit inner product, while pairs in different classes have $0$ inner product. The resulting equivalence classes of words form the {\em trace monoid} associated with $\Gamma$~\cite{youssefTuran}. 
    
    The operators $(\ell_1+\ell_1^\ast,\ldots,\ell_n+\ell_n^\ast)$ have semicircular marginals, and their mixed moments satisfy a generalized notion of independence known as {\em $\Gamma$-independence}~\cite{lambdaFreeness,gammaIndependenceDefinition,charlesworthCollins,CGVvH24}. We refer 
    to~\cite[Definitions 1.2 and 1.3]{BKNPW2026-Paper1} for a more detailed treatment of this example.
\end{example}

Together with the property $\ell_i^* \Omega=0$, the $\Gamma$-CCR uniquely define the joint law of the non-commutative random variables $\ell_1+\ell_1^\ast,\ldots,\ell_n+\ell_n^\ast$ with respect to the vacuum state. The diagonal entries $\Gamma_{ii}$ determine the marginal laws of each $\ell_i+\ell_i^*$, while the off-diagonal entries specify their mixed moments. In particular, $\ell_i+\ell_i^\ast$ is bounded whenever $\Gamma_{ii}<1$.

Our goal in this section and the next is to develop tools for bounding the norms of matrix pencils, i.e., operators of the form
\[ \mathcal{X} = \sum_{i = 1}^n A_i \otimes (\ell_i + \ell_i^*), \]
for the creation operators $\ell_1, \ldots, \ell_n$ on the $\Gamma$-Fock space, both for general $\Gamma$ and for the special cases in the examples above.

\subsection{General number operator}

We next generalize the concept of number
operator from Section~\ref{sec:gaussian} to the $\Gamma$-Fock space.

\begin{definition}
    \label{def:Gamma-N}
    Given $\Gamma\in [-1,1]_\sym^{n\times n}$, define the {\em $\Gamma$-number
    operator} to be
    \[
        N_\Gamma \defeq \sum_{i=1}^n \ell_i \ell_i^\ast\succeq 0\,,
    \]
    where $\ell_1, \ldots, \ell_n$ are the
    creation operators on the $\Gamma$-Fock space from Definition~\ref{def:gammaFock}. 
    For
    an integer $p\ge 0$, denote by $N_\Gamma|_p$ the restriction of $N_\Gamma$ to the direct sum of the $k$-particle
    spaces for $0\le k\le p$.
\end{definition}

For example, in the bosonic setting ($\Gamma\equiv 1$),
the number operator scales each basis vector
by its particle number. In the free
setting ($\Gamma\equiv 0$), the
number operator
is the projection onto $\Omega^\perp$.

Similarly to Section~\ref{sec:gaussian},
bounds on the norms of matrix pencils will require a bound on the norm of the associated number operator. We state two simple bounds that will be used below.

\begin{lemma}[Number of particles bound]\label{lem:boundParticle}
    For any $\Gamma\in [-1,1]_\sym^{n\times n}$ and $p\ge 0$, $\|N_\Gamma|_p\|\le p$.
\end{lemma}

\begin{proof}
    Since $N_\Gamma$ preserves each particle space,
    it suffices to consider an eigenvector $x\neq 0$
    in the $k$-particle space for some $k\le p$, with
    eigenvalue $\lambda$.
    Choose $y = \ell_{j_1}\cdots \ell_{j_k} \Omega$ maximizing
    $|\langle x,y\rangle_\Gamma|$ over all
    $j_1, \ldots, j_k\in [n]$ (the maximum is positive, as such vectors span the
    $k$-particle space). The
    effect of $N_\Gamma|_p$
    on $y$ is
    \begin{align*}
        N_\Gamma|_p y &= \sum_{i=1}^n \sum_{r:j_r=i} \left(\prod_{s<r} \Gamma_{i,j_s}\right) \ell_i \ell_{j_1}\cdots \ell_{j_{r-1}} \ell_{j_{r+1}} \cdots \ell_{j_k} \Omega \\
        &= \sum_{r=1}^k \left(\prod_{s<r} \Gamma_{j_r,j_s}\right) \ell_{j_r} \ell_{j_1}\cdots \ell_{j_{r-1}} \ell_{j_{r+1}} \cdots \ell_{j_k} \Omega\,,
    \end{align*}
    which is a sum over $k$ vectors of the form $\ell_{i_1}\cdots \ell_{i_k}\Omega$, each
    scaled by a coefficient of magnitude at most $1$.
    Since $N_\Gamma|_p$ is self-adjoint, this implies
    \[
        |\lambda \langle x,y\rangle_\Gamma| = |\langle N_\Gamma|_p x,y\rangle_\Gamma| = |\langle  x, N_\Gamma|_p y\rangle_\Gamma|\le k |\langle x,y\rangle_\Gamma|\,.
    \]
    So all eigenvalues of $N_\Gamma|_p$ are at most $p$; as $N_\Gamma|_p$ is positive semidefinite, this concludes the proof.
\end{proof}

\begin{lemma}[Clique bound]\label{lem:clique-bound}
    If $\Gamma$ is a simple graph, then $\|N_\Gamma\|\le \omega(\Gamma)$, where $\omega(\Gamma)$ denotes the maximum clique size in $\Gamma$.
\end{lemma}

\noindent
We note that this does not contradict our previous observation that $N_{\Gamma}$ is unbounded in the bosonic case $\Gamma \equiv 1$: in that case we also have $\Gamma_{i,i} = 1$, while here we assume $\Gamma_{i,i} = 0$.

Lemma~\ref{lem:clique-bound} also appears implicitly in the proof of~\cite[Theorem 1.5]{youssefTuran}. For completeness, we include a proof below.

\begin{proof}
    Let
    $\pi_i \defeq \ell_i \ell_i^*$. Since $\ell_i^* \ell_i = \mathbf{1}$ by the CCR, each $\pi_i$ is an
    orthogonal projection. If $i,j$ are
    adjacent, then by the CCR, one verifies that
    $(\ell_i \ell_j - \ell_j \ell_i)^* (\ell_i \ell_j - \ell_j \ell_i) = 0$,
    so $\ell_i$ and $\ell_j$ (and $\ell_i$ and $\ell_j^*$) commute, and hence so do $\pi_i$ and $\pi_j$. If $i,j$ are not adjacent in $\Gamma$, then $\pi_i \pi_j = \ell_i (\ell_i^* \ell_j) \ell_j^*=0$, so again $\pi_i$ and $\pi_j$ commute. For $S \subseteq [n]$, define
    \[
        \rho_S \defeq\prod_{i\in S} \pi_i \prod_{i\in [n]\setminus S} (\mathbf{1}-\pi_i).
    \]
    We have $\pi_i = \pi_i \prod_{j\neq i} (\pi_j + (\mathbf{1}-\pi_j)) = \sum_{S \subseteq [n]: i \in S} \rho_S$, and therefore
    $\sum_{i = 1}^n \pi_i = \sum_{S\subseteq [n]} |S| \rho_S$.
    Further, $\rho_S=0$ unless $S$ is a clique
    of $\Gamma$, so that 
    \[ N_\Gamma=\sum_{i = 1}^n \pi_i \preceq \omega(\Gamma) \sum_{S \subseteq [n]} \rho_S = \omega(\Gamma) \mathbf{1}, \]
    giving the desired bound.
\end{proof}

This proof of Lemma~\ref{lem:clique-bound}
only uses the CCR relations
abstractly.
Alternatively, one may also argue using
the trace monoid interpretation in Example~\ref{ex:traceMonoid}, as follows. On a basis
element $e_w$ corresponding to a word $w$, we have
$N_\Gamma e_w = \lambda e_w$,
where $\lambda$ is the number of letters
that are ``left-removable'' from $w$.
In particular, all of the basis elements are eigenvectors of $N_{\Gamma}$ and thus the eigenvalues are these values of $\lambda$ for each word $w$.
But, the set
of such left-removable letters from any word forms a clique in $\Gamma$, and so $\lambda\le \omega(\Gamma)$.

\section{Norm bounds under general commutation relations}\label{sec:general_norm_bounds}

Next, we use the
same approach as in Section~\ref{sec:gaussian} to derive
norm bounds for matrix pencils of operators satisfying
the general $\Gamma$-CCR, as defined in Section~\ref{sec:fock}. Our main theorems in this section are as follows.
\begin{itemize}
    \item Theorem~\ref{thm:q-proj} (Section~\ref{sec:deformedNCK})
    establishes an ``intrinsic $q$-Gaussian'' norm bound, akin to an improved NCK inequality with an error term associated to the $q$-commutation relations.
    \item Theorem~\ref{thm:variational} (Section~\ref{sec:lehner}) gives a variational norm bound for intrinsic freeness generalizing the upper bound of Lehner's formula to operators satisfying general $\Gamma$-CCR.
    \item Theorems~\ref{thm:deformed} and~\ref{thm:intrinsic2} (Section~\ref{sec:xfreeBound}) specialize this variational approach to the Gaussian matrix setting, obtaining a norm bound on Gaussian matrices given by a deformation of Lehner's variational principle and deriving Theorem~\ref{thm:strong-intrinsic} by similar calculations.
\end{itemize}

\subsection{A \texorpdfstring{$q$}{q}-deformed NCK inequality}
\label{sec:deformedNCK}

Theorem~\ref{thm:operator-NCK} shows how non-commutativity of coefficients can sharpen eigenvalue estimates of random matrices. Using the same approach, we can also study the behavior of matrices and operators with {\em partially commutative} structure. 
We consider the case where this structure
is encoded by the \emph{$q$-Gaussian distribution}.
Let $q\in [0,1)$ and consider the creation and annihilation operators $\ell$ and $\ell^\ast$ satisfying
\begin{align*}
    \ell^*\ell-q\ell\ell^* = 1\,.
\end{align*}
The $q$-Gaussian distribution is the unique distribution $\gamma_q$ whose moments are given by the vacuum states of powers of $\ell+\ell^*$. They can be characterized combinatorially as
\begin{align*}
    \int_{-\infty}^{\infty}t^p \,\mathrm d\gamma_q(t) = \inner{\Omega}{(\ell+\ell^*)^p\Omega}= \sum_{\pi\in P_2([p])}q^{c(\pi)}\,,
\end{align*}
where $P_2([p])$ is the set of matchings (equivalently, partitions into parts of size exactly 2) of the set $[p]$ and $c(\pi)$ is the number of crossings in such a matching $\pi$.
For $q\in [0,1)$, this distribution smoothly interpolates between the semicircular and Gaussian distributions (the limit as $q \to 1$), is symmetric and compactly supported, and its edge is given by
\begin{align*}
    \norm{\ell +\ell^*} = \frac{2}{\sqrt{1-q}}\,.
\end{align*}

The following theorem extends Theorem~\ref{thm:operator-NCK} by comparing a general $\Gamma$-CCR system with the corresponding operator model satisfying the $q$-Gaussian commutation relation. The error is controlled by the extent to which the matrix coefficients fail to transform the $\Gamma$-commutation relation into the $q$-commutation relation.

\begin{proposition}\label{prop:q-NCK}
    Let $A_1,\ldots, A_n \in M_d(\C)_\sa$, $\ell_1,\ldots, \ell_n$ be operators satisfying the $\Gamma$-CCR relations, and $\sL \defeq \sum_{i}A_i\otimes \ell_i$. Define the partial commutator
    \[ 
        C_q \defeq [\Gamma_{i,j}A_jA_i-qA_iA_j]_{i, j = 1}^n\in M_{dn}(\C)\,. 
    \] 
    Then
    \begin{align*}
        \sL^*\sL-q\sL\sL^*\preceq \sum_{i=1}^nA_i^2\otimes \mathbf{1} + \lambda_{\max}^+(C_q) I_d \otimes N_\Gamma\,. 
    \end{align*}
\end{proposition}

\begin{proof}
    The proof goes simply by expanding and
    using the $\Gamma$-CCR relations:
    \begin{align*}
        \sL^* \sL &= \sum_{i,j=1}^nA_{i}A_j\otimes \ell_i^*\ell_j \\
        &=\sum_{i=1}^nA_i^2\otimes \mathbf{1}+\sum_{i,j=1}^n\Gamma_{i,j}A_iA_j\otimes \ell_j\ell_i^*\\
        &=\sum_{i=1}^nA_i^2\otimes \mathbf{1}+q\sum_{i,j=1}^n A_iA_j\otimes \ell_i\ell_j^* + \sum_{i,j=1}^n(\Gamma_{i,j}A_{j}A_i-qA_iA_j)\otimes \ell_i\ell_j^*\\
        &=\sum_{i=1}^nA_i^2\otimes \mathbf{1}+q\sL\sL^* + \begin{bmatrix} I \otimes \ell_1 & \ldots & I\otimes \ell_n\end{bmatrix} ( C_q \otimes \mathbf 1)\begin{bmatrix} I \otimes \ell_1^*\\\vdots\\I \otimes \ell_n^*\end{bmatrix}\\
        &\preceq \sum_{i=1}^n A_i^2\otimes \mathbf{1}+q\sL\sL^*+\lambda_{\max}^+(C_q) \cdot I\otimes N_\Gamma\,.\qedhere
    \end{align*}
\end{proof}

\noindent As before, Proposition~\ref{prop:q-NCK} then leads to the following bounds on moments and norms.

\begin{theorem}[Intrinsic $q$-Gaussianity]\label{thm:q-proj}
    Let $A_1,\ldots, A_n \in M_d(\C)_\sa$, $\ell_1,\ldots, \ell_n$ be the creation operators on the $\Gamma$-Fock space, and $\sX = \sum_{i=1}^nA_i\otimes (\ell_i+\ell_i^*)$. Then for all $p\ge 1$ and $q\in [0,1)$,
    \begin{align*}
        \left(\tr\otimes\tau(\sX^{2p})\right)^{\frac{1}{2p}}\leq \frac{2}{\sqrt{1-q}}\cdot \sqrt{\norm{\sum_{i=1}^n A_i^2}+\lambda_{\max}^+(C_q)\cdot \norm{N_\Gamma|_p}}\,.
    \end{align*}
    Further, if $\Gamma_{ii}<1$ for all $i$, then we have
    \[ \|\mathcal{X}\| \leq \frac{2}{\sqrt{1-q}}\cdot \sqrt{\norm{\sum_{i=1}^n A_i^2}+\lambda_{\max}^+(C_q)\cdot \norm{N_\Gamma}}\,. \]
\end{theorem}
\noindent
We note that the factor outside the square root on the right-hand side is precisely the edge of the $q$-Gaussian distribution, while the first term inside the square root is $\sigma(X)^2$ for $X = \sum_{i = 1}^n g_i A_i$ the associated Gaussian series model.

\begin{proof}
    Let $\Pi_p$ be the projection onto the space of at most $p$ particles and $\sL_p \defeq (I\otimes\Pi_p)\sL(I\otimes \Pi_p)$. As in the proof of Corollary \ref{cor:nck-projection}, we
    have
    \[
        \tr \otimes \tau(\mathcal X^{2p}) = \frac 1d \sum_{i=1}^d \|\mathcal X^p(e_i \otimes \Omega)\|^2 = \frac 1d \sum_{i=1}^d \|(\mathcal L_p + \mathcal L_p^\ast)^p(e_i \otimes \Omega)\|^2\le \|\mathcal L_p + \mathcal L_p^\ast\|^{2p}\le 2^{2p}\|\mathcal L_p \mathcal L_p^\ast\|^{p}\,.
    \]
    Furthermore, by Proposition~\ref{prop:q-NCK}, 
    \begin{align*}
        \norm{\sL_p^*\sL_p}\leq \norm{\Pi_p\sL^*\sL\Pi_p}&\leq \norm{(I\otimes\Pi_p)\left(\sum_{i=1}^n A_i^2\otimes\mathbf{1}+q\sL\sL^*+\lambda_{\max}^+(C_q)\cdot I\otimes N_\Gamma \right)(I\otimes\Pi_p)}\\
        &\leq \norm{\sum_{i=1}^n A_i^2}+\lambda_{\max}^+(C_q)\norm{N_\Gamma|_p}+q\norm{(I\otimes\Pi_p)\sL\sL^*(I\otimes\Pi_p)}\\
        &= \norm{\sum_{i=1}^n A_i^2}+\lambda_{\max}^+(C_q)\norm{N_\Gamma|_p}+q\norm{\sL_p\sL_p^*}\,.
    \end{align*}
    Using  $\norm{\sL_p\sL_p^*} = \norm{\sL_p^*\sL_p}$ and rearranging the
    inequality concludes the proof of the first statement. 
    
    For the second
    statement, we use directly $\|\mathcal X\| \le 2 \|\mathcal L\mathcal L^\ast\|^{\frac 12}$ and proceed with
    the exact same argument as in the truncated setting, replacing $\Pi_p$ by $\mathbf 1$.
\end{proof}

\subsection{Lehner-type variational norm bounds beyond freeness}
\label{sec:lehner}

Next, we return to the question
of intrinsic freeness, 
now for the more general class of operators
satisfying the $\Gamma$-commutation relations
introduced in Section~\ref{sec:fock}. 
To that end, we introduce a variational refinement
of the operator inequalities developed 
above. 

Our model for such a bound is Lehner's variational formula~\cite{Leh99} (see~\eqref{eq:road_lehner}) for the norm of pencils of {\em freely independent} semicircular variables.
Here, we give a new proof of just the
{\em upper bound} in~\eqref{eq:road_lehner} which
also extends to arbitrary CCR relations.

\begin{proposition}\label{prop:road_master}
    In the setting of Proposition~\ref{prop:q-NCK}, we have for any $Z\succ 0$,
    \[
        \mathcal L + \mathcal L^\ast \preceq
        \left(Z^{-1}+\sum_{i=1}^n A_iZ A_i\right)\otimes\mathbf{1}
        +\lambda_{\max}^+(C_\Gamma(Z))I\otimes  N_\Gamma\,,
    \]
    where 
    \[ C_{\Gamma}(Z)\defeq [\Gamma_{i,j}A_jZA_i]_{i, j = 1}^n \]
    is the partial commutator scaled by $Z$. 
\end{proposition}

\noindent Applying Proposition~\ref{prop:road_master} to $-A_i$ yields a similar lower bound. In the case $\Gamma\equiv 0$, we recover the upper bound in Lehner's formula.

\begin{proof}
    Let us denote $\mathcal {Z} \defeq Z\otimes \mathbf{1}$. We have
    \[
        0\preceq \left(\mathcal Z^{\frac 12} \mathcal L  - \mathcal Z^{-\frac 12}\right)^\ast \left(\mathcal Z^{\frac 12} \mathcal L - \mathcal Z^{-\frac 12}\right) = \mathcal L^* \mathcal Z \mathcal L + \mathcal Z^{-1} - \mathcal L - \mathcal L^*\,,
    \]
    which implies $\mathcal L+\mathcal L^\ast\preceq \mathcal Z^{-1}+ \mathcal L^* \mathcal Z \mathcal L$.
    Applying the non-commutative relations,
    \[
        \mathcal L^* \mathcal Z \mathcal L = \sum_{i,j=1}^n A_i Z A_j \otimes \ell_i^* \ell_j = \sum_{i=1}^n A_i Z A_i \otimes \mathbf 1 + \sum_{i,j=1}^n \Gamma_{ij} A_i Z A_j \otimes \ell_j \ell_i^*\,.
    \]
    It only remains to bound the cross terms using the
    partial commutator:
    \begin{align*}
        \sum_{i,j=1}^n \Gamma_{ij} A_i Z A_j \otimes \ell_j \ell_i^* &= \sum_{i,j=1}^n \Gamma_{ij} A_j Z A_i \otimes \ell_i \ell_j^*\\
        &= \begin{bmatrix} I \otimes \ell_1 & \cdots & I\otimes \ell_n\end{bmatrix} (C_\Gamma(Z) \otimes \mathbf 1)\begin{bmatrix} I \otimes \ell_1^*\\\vdots\\I \otimes \ell_n^*\end{bmatrix}\\
        &\preceq \lambda_{\max}^+(C_\Gamma(Z)) I \otimes \sum_{i=1}^n \ell_i \ell_i^* = \lambda_{\max}^+(C_\Gamma(Z)) I \otimes N_\Gamma\,.\qedhere
    \end{align*}
\end{proof}

In the general situation where the number
operator
is unbounded, we may still apply
Proposition~\ref{prop:road_master} to 
obtain bounds on the
trace moments using the
same projection idea as in the previous
sections. 

\begin{theorem}[Variational intrinsic freeness bound]\label{thm:variational}
    Let $A_0,\ldots, A_n \in M_d(\C)_\sa$, $\ell_1,\ldots, \ell_n$ be the creation operators on the $\Gamma$-Fock space, and $\sX = A_0\otimes \mathbf 1 +\sum_{i=1}^nA_i\otimes (\ell_i+\ell_i^*)$. Then for any $p\ge 1$
    \[
        \left(\tr \otimes \tau (\sX^{2p})\right)^{\frac 1{2p}}\le \max_{\eps\in \{\pm 1\}} \inf_{Z\succ 0} \left\{\lambda_{\max} \left(\eps A_0 + \sum_{i=1}^n A_i Z A_i + Z^{-1}\right) + \lambda_{\max}^+(C_\Gamma(Z)) \norm{N_\Gamma|_p}\right\}\,.
    \]
\end{theorem}

\begin{proof}
    As in the proof of Theorem~\ref{thm:q-proj}, let $\Pi_p$ be the projection on the space of at most $p$ particles. Then we have $\tr \otimes \tau (\sX^{2p})\le \|\mathcal X_p\|^{2p}$, where $\mathcal X_p = A_0 \otimes \Pi_p + \mathcal L_p + \mathcal L_p^\ast$ and $\mathcal L_p = (I\otimes \Pi_p) \mathcal L (I\otimes \Pi_p)$. For each $\eps\in \{\pm 1\}$, applying Proposition~\ref{prop:road_master} to $\eps A_i$ and conjugating by $I\otimes \Pi_p$ gives
    \[
        \eps \mathcal X_p \preceq \left(\eps A_0 + Z^{-1} + \sum_{i=1}^n A_i Z A_i\right)\otimes \Pi_p +\lambda_{\max}^+(C_\Gamma(Z))I\otimes  N_\Gamma|_p\,.
    \]
    Taking the infimum over $Z$ and using $\|\mathcal X_p\| = \max_{\eps\in \{\pm 1\}} \lambda_{\max}(\eps \mathcal X_p)$ concludes the proof.
\end{proof}

\begin{proof}[Proof of Theorem~\ref{thm:generalizedLehnerIntro}]
    The proof is identical to the previous one, except that no truncation is needed. Applying Proposition~\ref{prop:road_master} to $\eps A_i$ for $\eps\in \{\pm 1\}$ directly gives
    \[
        \eps \mathcal X \preceq \left(\eps A_0 + Z^{-1} + \sum_{i=1}^n A_i Z A_i\right)\otimes \mathbf 1 +\lambda_{\max}^+(C_\Gamma(Z))I\otimes  N_\Gamma\,,
    \]
    and we conclude in the exact same way.
\end{proof}

Picking $Z$ to be a scalar multiple of the
identity in Theorem~\ref{thm:generalizedLehnerIntro} recovers a bound of the form of
Corollary~\ref{cor:nck-projection}.
More precisely, when $A_0=0$ and $Z=cI$ for an optimized
choice of $c>0$ and $N_{\Gamma}$ is bounded, we get
\[
    \|\mathcal{X}\| \le 2 \sqrt{\left\|\sum_{i=1}^n A_i^2\right\| + \lambda_{\max}^+(C_\Gamma(I)) \norm{N_\Gamma}}\,.
\]
For example, when $\Gamma$ is a simple graph and $\mathcal{X}$ is a matrix pencil of $\Gamma$-independent operators, we obtain an NCK-like inequality with an error term depending on the clique number $\omega(\Gamma)$.

\subsection{Specialization to Gaussian random matrices}
\label{sec:xfreeBound}

In this section, we specialize the variational bound of Theorem~\ref{thm:variational} to the case of
Gaussian random matrices,
\[
    X = A_0 + \sum_{i=1}^n g_i A_i\,.
\]
In that case, the partial commutator takes
the form
\[
    C_X(Z)\defeq [A_jZA_i]_{i, j = 1}^n\,. 
\]
Combining the weighted moment bound with
the trace method
gives the following deformation of Lehner's variational formula involving $C_X(Z)$.

\begin{theorem}[Deformed Lehner bound]\label{thm:deformed}
  Let $X$ and $C_X(Z)$ be as above. Then
  \[
      \EE\|X\|
      \lesssim \max_{\eps\in \{\pm 1\}}\inf_{Z\succ0}
      \left\{
          \lambda_{\max}\left(\eps A_0 +
              Z^{-1}+\sum_{i=1}^n A_iZA_i
          \right)
          +\log(2d)\lambda_{\max}(C_X(Z))
      \right\}\,.
  \]
\end{theorem}
\noindent
Comparing with Lehner's formula \eqref{eq:road_lehner} for $\|\mathcal{X}_{\free}\|$, we see that this is precisely Lehner's formula, but with an additional  term in the objective function.
\begin{proof}
  By Jensen's inequality, Theorem~\ref{thm:variational}, and the number of particles bound (Lemma~\ref{lem:boundParticle}), we have,
  for every $p\ge 1$,
  \begin{align*}
      \EE \|X\|
      &\leq d^{\frac 1 {2p}}
        \left(\EE \tr X^{2p}\right)^{\frac 1 {2p}}\\
      &\leq d^{\frac 1 {2p}} \max_{\eps\in \{\pm 1\}}\inf_{Z\succ0}
        \left\{
            \lambda_{\max}\left(\eps A_0 +
                Z^{-1}+\sum_i A_iZA_i
            \right)
            +p\lambda_{\max}(C_X(Z))
        \right\}.
  \end{align*}
  For $p=\lfloor\log(2d)\rfloor$, we have $d^{\frac 1 {2p}}\leq e$.
  Since
  $\Tr C_X(Z)=\sum_i\Tr(A_iZA_i)\geq0$,
  the maximal eigenvalue of $C_X(Z)$ is nonnegative. Replacing $p$ by $\log(2d)$ concludes the proof.
\end{proof}

Theorem~\ref{thm:variational} can also
be used to prove an upper
bound for the norm of 
Gaussian random matrices akin to our strengthened NCK inequality in Theorem~\ref{thm:strong-NCK}, but involving $\|\mathcal X_\free\|$
instead of $\sigma(X)$. We first prove the trace moment version; we will deduce Theorem~\ref{thm:strong-intrinsic} as a corollary. Recall that we define
$v'(X) \defeq \min\{\sigma(X), v(X)\}$.

\begin{theorem}\label{thm:intrinsic2}
    Let $X = A_0+\sum_i g_i A_i$ be a self-adjoint Gaussian random matrix.
    Define 
    \[
        \mathcal X_{\free} \defeq A_0 \otimes \mathbf 1 + \sum_{i=1}^n A_i\otimes s_i\,,
    \]
    where $s_1, \ldots, s_n$ are freely independent semicircular elements.
    Then for any $p\ge 1$,
    \begin{align*}
        \left(\EE \tr X^{2p}\right)^{\frac{1}{2p}}\leq \norm{\Xfree}+2v'(X)\sqrt{p}\,.
    \end{align*}
\end{theorem}

To prove Theorem~\ref{thm:intrinsic2}, we
use the following bound for the norm of the partial commutator as a function of the $v'$
parameter.

\begin{lemma}\label{lem:crossing}
    For any self-adjoint Gaussian random matrix $X = \sum_{i = 1}^n g_i A_i$ and $Z\succ 0$, we have $\Norm{C_X(Z)} \leq  v'(X)^2\norm{Z}$.
\end{lemma}
\begin{proof}
      We prove the upper bounds with $\sigma(X)$ and $v(X)$ separately.

      Define $A:\CC^d\to\CC^n\otimes\CC^d$ by
      $Ax=\sum_{i=1}^n e_i\otimes A_i x$.
      Then
      \begin{equation*}
          \|A\|^2
          =\|A^*A\|
          =\left\|\sum_{i=1}^n A_i^2\right\|
          =\sigma(X)^2.
          \label{eq:column-norm}
      \end{equation*}
      Let $S_{12}$ swap the two factors of $\CC^n\otimes\CC^n$.
      We have the factorization
      \begin{equation*}
          C_X(Z)
          =(I_n\otimes A^*)(S_{12}\otimes I_d)
          (I_{n^2}\otimes Z)(I_n\otimes A)\,.
          \label{eq:crossing-factor-sigma}
      \end{equation*}
      Indeed, the right-hand side sends $e_j\otimes x$ to
      $\sum_i e_i\otimes A_jZA_i x$.
      Since $S_{12}$ is unitary, taking norms gives
      $\|C_X(Z)\|\leq \sigma(X)^2\|Z\|$.

      For the other bound, define $V:\CC^n\to\CC^d\otimes\CC^d$ by
      $Ve_i=\vect(A_i)$, using the convention
      $\vect(A_i)=\sum_{a,b}(A_i)_{ab}e_a\otimes e_b$.
      By \eqref{eq:v-covariance},
      \begin{equation*}
          \|V\|=v(X)\,.
          \label{eq:coefficient-map-norm}
      \end{equation*}
      Let $F_{13}$ swap the first and third factors of $(\CC^d)^{\otimes 3}$.
      We claim that
      \begin{equation}
          C_X(Z)
          =(V^*\otimes I_d)F_{13}
          (I_d\otimes Z^\top\otimes I_d)(V\otimes I_d)\,,
          \label{eq:crossing-factor-v}
      \end{equation}
      where the transpose is taken in the basis used for vectorization.
      Applying the right-hand side to $e_j\otimes e_b$, the coefficient of
      $e_i\otimes e_a$ is
      \[
          \sum_{u,c}(A_j)_{au}Z_{uc}\overline{(A_i)_{bc}}
          =\sum_{u,c}(A_j)_{au}Z_{uc}(A_i)_{cb}
          =(A_jZA_i)_{ab}\,,
      \]
      where we used the self-adjointness of $A_i$.
      This proves \eqref{eq:crossing-factor-v}.
      Since $F_{13}$ is unitary and $\|Z^\top\|=\|Z\|$, taking norms gives
      $\|C_X(Z)\|\leq v(X)^2\|Z\|$.
      Combining the two bounds gives the result.
  \end{proof}

\begin{proof}[Proof of Theorem~\ref{thm:intrinsic2}]
    We first apply Theorem~\ref{thm:variational}
    with $\Gamma\equiv 1$:
    \begin{align}
        \left(\E \tr X^{2p}\right)^{\frac 1{2p}} &\le \max_{\eps\in \{\pm 1\}} \inf_{Z \succ 0} \left\{\lambda_{\max}\left(\eps A_0 + \sum_{i=1}^n A_i Z A_i + Z^{-1}\right) + \|C_X(Z)\| \|N_\Gamma|_p\|\right\}\nonumber\\
        &\le \max_{\eps\in \{\pm 1\}} \inf_{Z \succ 0} \left\{\lambda_{\max}\left(\eps A_0 + \sum_{i=1}^n A_i Z A_i + Z^{-1}\right) + p v'(X)^2 \|Z\|\right\}\,,\label{eq:lehnerFirst}
    \end{align}
    where we used Lemmas~\ref{lem:boundParticle} and~\ref{lem:crossing} in the second
    inequality.

    Given any $Z\succ 0$, consider
    $Z_t = (tI + Z^{-1})^{-1}$ for some parameter $t>0$ to be fixed later.
    Then we have $\|Z_t\|\le 1/t$ and $Z_t\preceq Z$, so
    \[
        \sum_{i=1}^n A_i Z_t A_i + Z_t^{-1} \preceq \sum_{i=1}^n A_i Z A_i + Z^{-1} + tI\,.
    \]
    Hence, any candidate $Z\succ 0$ to
    Lehner's formula~\eqref{eq:road_lehner} can
    be mapped to a candidate $Z_t\succ 0$ of~\eqref{eq:lehnerFirst} achieving objective value at most
    \[
        \lambda_{\max}\left(\eps A_0 + \sum_{i=1}^n A_i Z A_i + Z^{-1}\right) + t + \frac{p v'(X)^2} {t}\,.
    \]
    Optimizing the right-hand side over $t$ and taking the infimum
    over $Z$ yields
    the trace bound.
\end{proof}

\begin{proof}[Proof of Theorem~\ref{thm:strong-intrinsic}]
  By Jensen's inequality and Theorem~\ref{thm:intrinsic2}, for every integer $p\geq1$,
  \[
      \EE \Norm{X}\le d^{\frac 1 {2p}} \left(\E \tr X^{2p}\right)^{\frac 1{2p}}\le d^{\frac 1 {2p}} \left(\|
      \mathcal X_\free\| + 2 v'(X) \sqrt {p}\right).
  \]
  Applying Lemma~\ref{lem:optimization-free} with $a = \|\mathcal X_\free\|$ and $b = v'(X)$, and using $
  \log(2d)\leq2\log d$, gives
  \begin{align*}
      \EE \Norm{X}
      &\le \|\mathcal X_\free\| + 8\,\|\mathcal X_\free\|^{1/3} v'(X)^{2/3} (\log d)^{1/3} + 8\,v'(X)(\log d)^{1/2} \\
      &\le \|\mathcal X_\free\|\left(1 + 8\,\left(\frac{v(X)^2}{\sigma(X)^2}\log d\right)^{1/2} +
      8\,\left(\frac{v(X)^2}{\sigma(X)^2}\log d\right)^{1/3}\right),
  \end{align*}
  where the last step uses $\sigma(X)\le \|\mathcal X_\free\|$ and $v'(X)\le v(X)$.
\end{proof}

\section{Applications of norm bounds}
\label{sec:appl}

\subsection{Gaussian Cayley matrices}
\label{sec:cayley}

Given a finite group $G$ on $d$ elements, let $\{\rho_x\}_{x\in G}$ denote the left regular representation of $x\in G$ (i.e., $\rho_x$ is the $d\times d$ permutation matrix corresponding to left multiplication by $x$).
Consider the asymmetric Gaussian random matrix $Y^{(0)} = \sum_{x\in G}g_x\rho_x$, where $g_x$ are i.i.d.\ drawn from $\NNN(0,1)$, and define its Hermitian dilation
\[ X^{(0)} = \left[\begin{array}{cc} 0 & Y^{(0)} \\ Y^{(0)^{\top}} & 0 \end{array}\right]\,. \]
Note that $\|Y^{(0)}\| = \|X^{(0)}\|$, so we may consider these interchangeably, but we focus on using our tools on $X^{(0)}$ since it is Hermitian.

As we will see, it is helpful to apply a simple reduction before applying matrix concentration inequalities.
The all-ones vector $\chi$ of suitable dimension is an eigenvector of both $X^{(0)}$ and $Y^{(0)}$ with eigenvalue $\lambda = \sum_{x \in G}g_x$, whose law is $\mathcal{N}(0, |G|)$.
Thus, we may further define modified versions of these matrices where we subtract off this component,
\begin{align*}
    Y &\colonequals Y^{(0)} - \frac{\lambda}{|G|} \chi\chi^{\top}, \\
    X &\colonequals \left[\begin{array}{cc} 0 & Y \\ Y^{\top} & 0 \end{array}\right]\,.
\end{align*}
Then $X$ is again the Hermitian dilation of $Y$ and we have $\|X\| = \|Y\|$, and further for the original matrices we have
\[ \|Y^{(0)}\| = \max\left\{\|X\|, |\lambda|\right\}, \]
where the second term is with high probability $\widetilde{O}(|G|^{1/2})$ and will turn out to be essentially negligible in our calculations.
We thus focus on bounding $\|X\|$ accurately.

While the spectral norm of $X$ is known to depend heavily on the group structure and the dimensions of its irreducible representations~\cite{BKMZ24-Cayley}, the parameters $\sigma(X)$ and  $v(X)$ depend only on $d = |G|$: straightforward calculations give that $\sigma(X) = \sqrt{d}$ and $v(X) = \sqrt{2d}$.
Thus the original NCK inequality (Theorem~\ref{thm:nck}) gives the estimate
\begin{equation}
\EE\|X\| \lesssim \sqrt{d \log d} \label{eq:cayley-nck}
\end{equation}
while the improved NCK inequality of \cite{BBvH23} (Theorem~\ref{thm:bbvh}) gives the same with an inferior logarithmic factor.
The estimate in \eqref{eq:cayley-nck} is tight for some groups, for instance, for the cyclic group $G = \mathbb{Z} / d$.
However, for other groups we have $\EE\|X\| \asymp \sqrt{d}$.
In fact, this was precisely the motivation for the study of these matrices in~\cite{BKMZ24-Cayley}: it is one where the tools of~\cite{BBvH23} fail, while the spectrum can be understood with representation theory (see~\cite{BKMZ24-Cayley} for details). 
We show here that our parameter $\kappa(X)$ can witness the representation-theoretic properties that determine the spectrum of $X$. 

\begin{proposition}
    \label{prop:cayley-kappa}
    In the above setting, let $r_{\min}(G)$ be the smallest dimension of a nontrivial complex irreducible representation of $G$. Then,
      \[
          \kappa(X)=\sqrt{\frac{d}{r_{\min}(G)}}
          =\frac{\sigma(X)}{\sqrt{r_{\min}(G)}}\,.
      \]
  \end{proposition}

  \begin{proof}
      Set $P=d^{-1}\chi\chi^\top$.
      We have $X = \sum_{x \in G} g_x A_x$ for coefficient matrices
      \[
          A_x=\begin{bmatrix}
              0&\rho_x-P\\
              \rho_x^\top-P&0
          \end{bmatrix}\,.
      \]
      Since $\rho_xP=P\rho_x=P$, permuting the tensor factors
      identifies $C(X)=[A_yA_x]_{x,y\in G}$ with the direct sum $T_+ \oplus T_-$, where
      \[
          T_+=[\rho_{yx^{-1}}-P]_{x,y\in G}\,,
          \qquad
          T_-=[\rho_{y^{-1}x}-P]_{x,y\in G}\,.
      \]
      The permutation $x\mapsto x^{-1}$ of the block indices
      swaps $T_+$ with $T_-$.
      Thus these two have the same largest eigenvalue and $\kappa(X)^2=\lambda_{\max}(T_+)$.

      Decomposing the regular representation into irreducible representations over $\CC$ shows that, after conjugation by a suitable unitary matrix,
      $T_+$ is a direct sum of a zero block and the matrices
      \[
          T_\pi=[\pi(y)\pi(x)^*]_{x,y\in G}\,,
      \]
      each repeated $\dim\pi$ times, where $\pi$ ranges over the
      nontrivial irreducible unitary representations of $G$ (see, e.g., \cite[Lemma 7]{BKMZ24-Cayley} for a similar decomposition).
      We show that $\lambda_{\max}(T_\pi)=d/\dim\pi$.

      Fix such a representation $\pi$ of dimension $r$, and let
      $V_\pi\in M_{r^2,d}(\CC)$ have columns
      $\operatorname{vec}(\pi(x))$, indexed by $x\in G$.
      The Peter-Weyl decomposition gives $V_\pi V_\pi^*=(d/r)I_{r^2}$.
      Also, by Schur's lemma,
      \[
          (T_\pi^2)_{x,z}
          =\sum_{y\in G}\pi(y)\pi(x)^*\pi(z)\pi(y)^*
          =\frac dr\Tr(\pi(x)^*\pi(z))I_r\,.
      \]
      Consequently,
      \[
          T_\pi^2=\frac dr(V_\pi^*V_\pi)\otimes I_r\,.
      \]
      Every nonzero eigenvalue of $V_\pi^*V_\pi$ equals $d/r$,
      so every nonzero eigenvalue of $T_\pi$ equals $d/r$ or $-d/r$.
      Since $\Tr(T_\pi)=dr>0$, we have
      $\lambda_{\max}(T_\pi)=d/r$.
      Taking the maximum over the nontrivial irreducible
      representations proves the claim.
  \end{proof}

In particular, over a sequence of groups $G = G_n$ such that $r_{\min}(G_n) \gg \log |G_n|$,  giving rise to a sequence $X = X_n$, our Theorem~\ref{thm:strong-NCK} together with Proposition~\ref{prop:cayley-kappa} establishes $\EE\|X\| \asymp \sigma(X) = \sqrt{|G|}$, showing that the logarithmic factor incurred by the other NCK-type inequalities is spurious in this case.
Such groups are called (sometimes subject to other quantitative assumptions about how $r_{\min}(G)$ depends on $|G|$) \emph{quasirandom groups} in the literature; see \cite{BKMZ24-Cayley} for further discussion on this connection.

\subsection{NCK-type inequalities for \texorpdfstring{$\Gamma$}{Γ}-independent systems}

Despite the remarkable success of the theory
of intrinsic freeness, there are important cases
of Gaussian random matrices for which $\mathcal X_\free$ fails to describe the
spectrum and the norm --- for example,
for
the so-called {\em Kikuchi matrices}, which arise in
the analysis of algorithms for tensor PCA~\cite{WeinElAlaouiMoore2019,BCSvH24,kothari2025smoothtradeofftensorpca}. The parallel work~\cite[Section 1.2]{BKNPW2026-Paper1}
suggests that the right operator proxy for
these matrices is not $\mathcal X_\free$, but rather
$\mathcal X_\Gamma\defeq \sum_i A_i \otimes s_i$ 
for an appropriate graph~$\Gamma$, where the $s_i$ are
$\Gamma$-independent semicircles (Example~\ref{ex:traceMonoid}).

This example motivates the development of a theory
for 
``intrinsic $\Gamma$-independence'' that would
generalize the notion of intrinsic freeness.
A natural first step for this program is to understand the norm
of the pencils $\mathcal X_\Gamma$, just as Lehner's formula
(and the coarser bound $
\|\mathcal X_\free\|\le 2\sigma(X)$\,) has allowed the development of
intrinsic freeness.

Such operator-valued non-commutative Khintchine
inequalities depending on combinatorial
properties of the
graph $\Gamma$ have independently emerged 
in recent years. To compare these results, 
let us fix some notation for the rest of this section. Let $A_1, \ldots, A_n\in M_d(\CC)_\sa$ and set
\[
    \sigma^2 \defeq \Norm{\sum_{i=1}^n A_i^2}\,.
\]
Let $s_1, \ldots, s_n$ be $\Gamma$-independent 
semicircles (i.e., $s_i=\ell_i+\ell_i^\ast$,
where $\ell_i$ are solutions to the $\Gamma$-CCR for a simple unweighted graph $\Gamma$). Let $\mathcal X \defeq \sum_i A_i \otimes s_i$. Finally, we
denote by $\lambda_1(\Gamma)\ge \ldots \ge \lambda_n(\Gamma)$
the eigenvalues of $\Gamma$ and $\omega(\Gamma)$
its maximum clique size. 

In this notation, Collins and Miyagawa
show that $\|\mathcal X\|\le 2\sigma \sqrt{\lambda_1(\Gamma)+1}$~\cite{collinsEigenvalueBound}; and
Oliveira Santos, Tripathi, and Youssef show that
$\|\mathcal X\|\le 2\sigma \sqrt{\omega(\Gamma)}$~\cite{youssefTuran}. The latter bound is stronger, since
$\omega(\Gamma)\le \lambda_1(\Gamma)+1$.
Meanwhile, combining Theorem~\ref{thm:q-proj}
with Lemma~\ref{lem:clique-bound} directly gives:

\begin{corollary}\label{cor:Gamma-nck}
    Let $C_\Gamma \defeq [\Gamma_{ij} A_j A_i]_{1\le i,j\le n}$ be the partial commutator of $A_1, \ldots, A_n$. Then
    \[
        \|\mathcal X\|\le 2\sqrt{\sigma^2 + \lambda_{\max}^+(C_\Gamma) \omega(\Gamma)}\,.
    \]
\end{corollary}

Since $\lambda_{\max}^+(C_\Gamma)\le \sigma^2$ (by the same argument as in Proposition~\ref{prop:comparisonKappa}), the bound in
Corollary~\ref{cor:Gamma-nck} is always at least as good,
up to a universal constant factor, as the
previous estimates.
Further, whenever $\lambda_{\max}^+(C_\Gamma)
\ll {\sigma^2}/{\omega(\Gamma)}$, 
our bound is $(2+o(1))\sigma$, which
substantially improves over these.

Another estimate that the preceding inequalities fail to capture is the following, which plays a key role in proving that independent Kikuchi matrices converge to a $q$-Gaussian system~\cite{BKNPW2026-Paper1}. It follows directly from Theorem~\ref{thm:q-proj}, applied with scalar coefficients $1/\sqrt n$.

\begin{corollary}\label{cor:Gamma-scalar}
    We have
    \[
        \norm{\frac{1}{\sqrt{n}}\sum_{i=1}^n s_i}\leq \inf_{q\in [0,1)}\frac 2 {\sqrt{1-q}} \sqrt{{1+\frac {\lambda_{\max}^+(\Gamma-{q}\chi\chi^{\top})\omega(\Gamma)} n}}\,.
    \]
\end{corollary}

\noindent The $q=0$ case of this inequality also appears in~\cite[Corollary 3.5]{collinsEigenvalueBound}.

\subsection{Ground state energy of the SYK model}\label{sec:syk}

The Sachdev-Ye-Kitaev (SYK) model is a model from
quantum physics that has attracted significant attention; we refer the reader to~\cite{PlumaRodriguez19,SYKPravesh} for background.
Let $0\le k\le n$ be even integers. The problem of finding the ground state energy of
the SYK model reduces to computing the largest eigenvalue of the
random matrix
\begin{equation*}
    H_{n,k} = {\binom n k}^{-\frac 12} \sum_{S\in \binom {[n]} k} g_S \psi_S\,,
\end{equation*}
where the $g_S$'s are independent standard Gaussians, and $\psi_S\in M_{2^{n/2}}(\CC)_\sa$  satisfy the relations
\begin{equation}
    \psi_S^2 = I\,,\qquad  \psi_S \psi_T = (-1)^{|S\cap T|} \psi_T \psi_S\,.\label{eq:syk}
\end{equation}
As $n\to\infty$ for $k=k(n)$
with $k^2/n \to \alpha\in [0,\infty)$, the spectral distribution 
of $H_{n,k}$ converges to the $e^{-2\alpha}$-Gaussian distribution~\cite{spectrumSYK}. The limiting spectral distribution is Gaussian if $k\ll \sqrt n$, semicircular if $k\gg \sqrt n$, and $q$-Gaussian in 
the intermediate regime.
Very recently, Basu, Kothari and Midha~\cite{SYKPravesh} proved sharp asymptotics for the norm of
$H_{n,k}$ in the regime $1\ll k\ll \sqrt n$.
As a corollary of Theorem~\ref{thm:q-proj}, we
obtain a new proof of their upper bound.

\begin{corollary}\label{cor:syk}
    For every even $2\le k< \frac {\sqrt n} {4}$, we have
    \[
        \E \Norm{H_{n,k}}\le \frac {\sqrt {2n}} k + O(1)\,.
    \]
\end{corollary}

\begin{proof}
    Think of $H_{n,k}$
    as a pencil in Gaussian variables, with the
    matrices $\psi_S$ playing the role of the coefficients. By Theorem~\ref{thm:q-proj} with $A_S = {\binom n k}^{-1/2} \psi_S$, we have, for every even integer $p$
    and $q\in [0,1)$,
    \begin{equation}
        \left(\E \tr H_{n,k}^p\right)^{\frac 1p}\le \frac 2 {\sqrt{1-q}} \sqrt{1 + \frac {p\lambda_{\max}^+(C_q)} {\binom n k}}\,,\qquad \text{where }C_q \defeq \left[(\psi_T \psi_S - q \psi_S \psi_T)_{S,T\in \binom {[n]} k}\right]\,.\label{eq:sykBasic}
    \end{equation}
    Let $\Gamma_{n,k}$ be the $\binom n k\times \binom n k$-dimensional matrix with entries $\Gamma_{n,k}[S,T] \defeq (-1)^{|S\cap T|}$.
    The commutation relations~\eqref{eq:syk} give
    the decomposition
    \begin{align*}
        C_q &= \operatorname{diag}(\{\psi_S\}_S) \left((\Gamma_{n,k} - q\chi \chi ^\top)\otimes I_{2^{n/2}}\right) \operatorname{diag}(\{\psi_S\}_S)\\
        &= \operatorname{diag}(\{\psi_S\}_S) \left((\Gamma_{n,k} - q\chi \chi ^\top)\otimes I_{2^{n/2}}\right) \operatorname{diag}(\{\psi_S\}_S)^{-1}\,,
    \end{align*}
    where the second equality follows from $\psi_S = \psi_S^{-1}$ (from~\eqref{eq:syk}). Therefore, $C_q$ has the same spectrum as the inner matrix, and $\lambda_{\max}(C_q) = \lambda_{\max}(\Gamma_{n,k} - q\chi\chi^\top)$.
    The remainder of the proof is a simple estimate
    on the top two eigenvalues of the matrix $\Gamma_{n,k}$.

    \begin{lemma}\label{lem:spectralBound}
        The top eigenvalue of $\Gamma_{n,k}$ is
        \[
            \binom n k\cdot q_{n,k}\,,\qquad \text{with } q_{n,k} = 1 - \frac {2k^2} n + O\left(\frac {k^4} {n^2}\right)\,,
        \]
        with eigenvector $\chi$, and
        \[
            \lambda_{\max}\left(\Gamma_{n,k} - q_{n,k}\chi \chi^\top\right) \lesssim \frac {k^2} {n^2} \binom n k\,.
        \]
    \end{lemma}

    \noindent See~\cite[Lemma 5.1]{SYKPravesh} for a proof of Lemma~\ref{lem:spectralBound}. Plugging this choice of $q_{n,k}$ into~\eqref{eq:sykBasic}, we obtain
    \[
        \left(\E \tr H_{n,k}^p\right)^{\frac 1p}\le \frac {\sqrt{2n}} {k} \left(1 + O\left(\frac {k^2} n\right)\right) \sqrt{1 + O\left(\frac {pk^2} {n^2}\right)}\,.
    \]
    Picking $p = \frac {n^{3/2}} k$ and noting that $H_{n,k}$ has dimension $2^{n/2}$,
    \[
        \E \Norm{H_{n,k}}\le 2^{\frac {n}{2p}} \left(\E \tr H_{n,k}^p\right)^{\frac 1p}\le \frac {\sqrt{2n}} k + O(1)\,.\qedhere
    \]
\end{proof}
 
\begin{remark}[Comparison with \cite{SYKPravesh}]\label{rem:syk}
In summary, our proof of Corollary~\ref{cor:syk} proceeds by
representing $H_{n,k}$ as the matrix pencil
\begin{equation}
    {\binom n k}^{-\frac 12}\sum_S \psi_S \otimes (\ell_S+\ell_S^\ast) \,,\label{eq:bosonicSYK}
\end{equation}
where $\ell_S$ are the bosonic creation operators,
and then applying our non-commutative Khintchine inequality in the bosonic setting, with the matrices $\psi_S$ as coefficients. Instead,~\cite{SYKPravesh} first shows that
the operator~\eqref{eq:bosonicSYK}
has the same distribution as the operator
\begin{equation}
    {\binom n k}^{-\frac 12}\sum_S (a_S+a_S^\ast)\,,\label{eq:gammaSYK}
\end{equation}
where $a_S$ are the creation operators
for the $\Gamma_{n,k}$-CCR by comparing their moments (where $\Gamma_{n,k}$ is defined in the above proof). Then, they analyze the trace moments of~\eqref{eq:gammaSYK} using techniques closely related to those in our proof of Theorem~\ref{thm:q-proj}. We note that applying Theorem~\ref{thm:q-proj} directly to~\eqref{eq:gammaSYK} also recovers Corollary~\ref{cor:syk} without any combinatorial analysis of the SYK moments.
\end{remark}

\input{main.bbl}

\appendix

\section{Scalar optimization estimates}

We include the elementary estimates used to pass from trace moments to the operator norm.
The constants are not optimized.

\begin{lemma}
  \label{lem:optimization-free}
  For every $d \geq 2$ and $a,b \geq 0$,
  \[
      \inf_{p \ge 1} d^{\frac 1 {2p}}(a + 2b\sqrt p)
      \leq a + 5\left[(ab^2\log(2d))^{\frac 13} + b\log^{\frac 12}(2d)\right]\,.
  \]
\end{lemma}

\begin{proof}
  Set $\ell = \log(2d)>1$. If $b=0$, take $p\to\infty$; otherwise, set $p \defeq \left\lceil \ell + \left(\frac{a\ell}{b}\right)^{\frac 23}\right\rceil$.
  As $p \geq \ell$, the bound $e^u \leq 1 + 2u$ for $0 \leq u \leq \frac 12$ gives $d^{\frac 1 {2p}} \leq e^{\frac \ell {2p}} \leq 1 + \frac{\ell}{p}$.
  Therefore,
  \begin{align*}
      d^{\frac 1 {2p}}(a + 2b\sqrt p)
      &\le a+ \frac{a\ell}{p} + 2b\sqrt p + \frac{2b\ell}{\sqrt p}\\
      &\le a + (a \ell)^{\frac 13} b^{\frac 23} + 2b\sqrt{2\ell + 2\left(\frac{a\ell}{b}\right)^{\frac 23}} + 2b\sqrt \ell\,.
  \end{align*}
  Using $\sqrt{x+y}\le \sqrt{x} + \sqrt{y}$
  and bounding $2+2\sqrt{2}<5$ concludes the proof.
\end{proof}

\begin{lemma}
    \label{lem:optimization-scalar}
    For every $d \geq 2$ and $0 \leq \kappa \leq \sigma$,
    \begin{equation*}
        \inf_{p\ge 1} d^{\frac 1 {2p}}\sqrt{\sigma^2 + p\kappa^2}
        \leq \sigma + 2\kappa \log^{\frac 12}(2d)\,.
        \label{eq:optimization-scalar}
    \end{equation*}
\end{lemma}

\begin{proof}
  Let $\ell \defeq \log(2d)> 1$. If $\kappa=0$, take $p\to\infty$; otherwise, set
  $p \defeq \left\lceil \ell+\frac{\sigma\sqrt \ell}{\kappa}\right\rceil$.
  By the same argument as in the previous lemma,
  $d^{\frac 1 {2p}}\le 1+\frac \ell p$.
  Moreover, since $\kappa\le \sigma$ and $\ell>1$,
  \[
    \sigma^2+p\kappa^2
      \le \sigma^2+\ell\kappa^2+\sigma\kappa\sqrt\ell+\kappa^2
      \le (\sigma+\kappa\sqrt\ell)^2\,.
  \]
  Finally, $\frac{\ell}{p}(\sigma+\kappa\sqrt\ell)
      \le {\kappa\sqrt\ell}{}$, so
  \[
      d^{\frac1{2p}}\sqrt{\sigma^2+p\kappa^2}
      \le \left(1+\frac{\ell}{p}\right)
          (\sigma+\kappa\sqrt\ell)
      \le \sigma+2\kappa\sqrt\ell\,.\qedhere
  \]
\end{proof}

\end{document}

%% file: main.bbl
\newcommand{\etalchar}[1]{$^{#1}$}